\documentclass{article}

\usepackage[english]{babel}

\usepackage[letterpaper,top=2cm,bottom=2cm,left=3cm,right=3cm,marginparwidth=1.75cm]{geometry}

\usepackage{amsmath}
\usepackage{amssymb}
\usepackage{amsthm}
\usepackage{graphicx}
\usepackage[colorlinks=true, allcolors=blue]{hyperref}
\usepackage{comment}
\usepackage[all]{xy}

\newcommand{\im}{\mathrm{im}\,}
\newcommand{\R}{\mathbb{R}}

\newcommand{\persmod}{\mathbb{V}}
\newcommand{\PH}{\mathrm{PH}}
\newcommand{\fib}{\mathrm{PH}^{-1}(D)}
\newcommand{\fibf}{\mathrm{PH}_f^{-1}(D)}
\newcommand{\fibT}{\mathrm{MT}^{-1}(\cellMTree)}
\newcommand{\MT}{\mathrm{MT}}
\newcommand{\LocalMin}{M}
\newcommand{\cellMTree}{T}
\newcommand{\leftChild}{L}
\newcommand{\rightChild}{R}
\newcommand{\Tree}{T}
\newcommand{\tree}{X}
\newcommand{\subtree}{Y}

\newcommand{\epi}{\mathrm{epi}}
\newcommand{\LCA}{\mathrm{LCA}}
\newcommand{\conf}{\mathrm{Conf}}

\newcommand{\critconf}{\mathrm{Conf}_{\mathrm{Crit}}}
\newcommand{\minconf}{\mathrm{Conf}_{\mathrm{Min}}}

\newcommand{\Labelleaf}{l}
\newcommand{\Labelnode}{v}
\newcommand{\conv}{\mathrm{Conv}}
\newcommand{\ShortPath}{\mathrm{ShortPath}}
\newcommand{\ConnectedComp}{\Omega}
\newcommand{\homotopy}{\mathrm{H}}
\newcommand{\LCAmatrix}{\mathcal{M}}
\newcommand\isomto{\stackrel{\sim}{\smash{\longrightarrow}\rule{0pt}{0.4ex}}}

\newtheorem{theorem}{Theorem}
\newtheorem{definition}[theorem]{Definition}
\newtheorem{lemma}[theorem]{Lemma}
\newtheorem{proposition}[theorem]{Proposition}
\newtheorem{corollary}[theorem]{Corollary}
\newtheorem{remark}{Remark}
\newtheorem*{example}{Example}
\newtheorem*{proposition*}{Proposition}

\makeatletter
\newtheorem*{rep@proposition}{\rep@title}
\newcommand{\newrepproposition}[2]{
\newenvironment{rep#1}[1]{
 \def\rep@title{#2 \ref*{##1}}
 \begin{rep@proposition}}
 {\end{rep@proposition}}}
\makeatother
\newrepproposition{proposition}{Proposition}

\title{The fiber of persistent homology for trees}
\author{David Beers, Jacob Leygonie}

\begin{document}
\maketitle

\begin{abstract}
Consider the space of continuous functions on a geometric tree~$X$ whose persistent homology gives rise to a finite generic barcode~$D$. We show that there are exactly as many path connected components in this space as there are merge trees whose barcode is~$D$. We find that each component is homotopy equivalent to a configuration space on~$X$ with specialized constraints encoded by the merge tree. For barcodes~$D$ with either one or two intervals, our method also allows us to compute the homotopy type of this space of functions.
\end{abstract}

\section{Introduction}
\subsection{Motivation}
Persistent Homology ($\PH$) is a computable descriptor from Topological Data Analysis (TDA) which summarises complex geometric data. More precisely, the persistence map, denoted~$\PH$, takes as input a topological space~$X$ equipped with a real valued function~$f:X\to \R$ and returns a multiset of intervals in the real line called a {\em barcode}, which encodes the topological variations across the sublevel sets of~$f$. In a wide range of situations, persistent homology is robust to perturbations of the input data~\cite{cohen2007stability}, which is one of the key reasons for its successful application to problems in data science, e.g. in neuroscience~\cite{bendich2016persistent}, material sciences~\cite{hiraoka2016hierarchical}, shape recognition~\cite{li2014persistence}, and machine learning~\cite{chen2019topological}.

Complementarily, it is natural to ask how decisively~$\PH$ distinguishes distinct input functions~$f$. Equivalently, we may ask which functions give rise to the same barcode~$D = \PH(f)$. This inverse problem formally translates into studying the fiber~$\fib$ over a target barcode~$D$. The topological and geometric properties of~$\fib$ strongly depend on the underlying space~$X$ and on the space~$\mathcal{F}$ of functions~$f:X\rightarrow \R$ on which persistent homology is defined, which can be for instance the space of {\em filter functions} (or one of its subspaces) when~$X$ is a simplicial or CW complex, the space of Morse functions when~$X$ is a smooth manifold, or simply of continuous functions when~$X$ is merely a topological space.  

For filter functions on a simplicial complex, it was observed in~\cite{leygonie2022fiber} that~$\fib$ has the structure of a finite polyhedral complex. This polyhedral structure was exploited in~\cite{leygonie2021algorithmic} to design an algorithm for computing the homology groups of~$\fib$, and this algorithm was demonstrated on a menagerie of small examples. When~$\mathcal{F}$ is the subspace of filter functions determined by their values on vertices, it was shown that every connected component of~$\fib$ is contractible when~$X$ is a simplicial decomposition of the unit interval~\cite{cyranka2020contractibility}, and homotopy equivalent to a circle when~$X$ is instead a simplicial decomposition of the circle~\cite{mischaikow2021persistent}. 

Cases where~$X$ is non-discrete have also been investigated. For {\em Morse-like} continuous functions on the unit interval, the number of path components of~$\fib$ was computed for generic barcodes~\cite{curry2018fiber}. For Morse functions on the~$2$-sphere~$\mathbb{S}^2$ obtained by composing an embedding of~$\mathbb{S}^2$ in~$\R^3$ with the vertical projection, the tools developed in~\cite{catanzaro2020moduli} motivated conjectures on the number of connected components of~$\fib$. For general Morse functions on an arbitrary smooth compact manifold, it was proven in~\cite{leygonie2022fiber} that the groups of diffeomorphisms of~$X$ isotopic to the identity defines an action on~$\fib$ which is transitive on each connected component. This allowed computing the homotopy type of path components of~$\fib$ for Morse functions on~$1$-dimensional and~$2$-dimensional oriented manifolds.

However, the tools developed in the above literature do not adapt easily to continuous functions on a topological space~$X$ that is not a manifold. In~\cite{mischaikow2021persistent}, it was observed that when~$X$ is a star-like tree and~$D$ is the specific barcode that has only one finite interval, then the path connected components of~$\fib$ are wedges of circles. In this work, we analyze~$\fib$ in the case of an arbitrary generic barcode~$D$, for continuous functions on an arbitrary geometric tree. 

The case of a tree is of particular interest as it is frequently encountered in applications of persistent homology to neuroscience, e.g. for analyzing neuronal morphologies~\cite{kanari2018topological, kanari2019objective} and brain functionalities~\cite{bendich2016persistent}. In fact, a few other related inverse problems for topological descriptors on a tree have already been studied. For instance, statistical and algorithmic inverses of the Topological Morphology Descriptor (TMD) have been described in~\cite{curry2021trees, kanari2020trees}. Another example is the study of the realization problem for barcodes of functions on a tree, which have been investigated in~\cite{johnson2022merge,liu2020realization}.

\subsection{Contributions and outline of contents}
In this work we study the case when~$X$ is the geometric realization of a tree (geometric tree for short),~$\mathcal{F}$ is the space of continuous functions on~$X$, and~$D$ is a finite generic barcode. For this reason,~$X$ denotes any geometric tree for the remainder of the introduction. Our analysis relies upon the fact that for functions~$f$ on~$X$, the persistence map factors in the following way:
$$\xymatrix@1{\PH:  f \ar@{|->}[rr]^-{\MT} && \cellMTree \ar@{|->}[rr] && D }.$$

Here, the intermediate object~$\cellMTree$ is a topological space called the \emph{merge tree} of~$f$, which describes how the connected components of the sublevel sets~$f^{-1}(-\infty,t]$ appear and join together as~$t$ varies. Hence, to characterize the fiber of persistent homology in this setting, we can instead characterize the space of functions with a given merge tree and the space of merge trees that map to~$D$. The main contributions of this work are:
\begin{itemize}
    \item In Theorem~\ref{theorem_merge_tree_are_cellulars}, we provide sufficient conditions for a merge tree derived from a function~$f$ on a topological space to have a cellular structure.
    
    \item In Theorem~\ref{theorem_homotopy_equivalence_configuration_spaces}, we show that~$\MT^{-1}(\cellMTree)$ is homotopy equivalent to a constrained version of the configuration space of~$n$ points on~$X$, denoted~$\conf(X,\cellMTree)$, where the points must satisfy additional constraints determined by~$\cellMTree$.
    \item In Theorem \ref{theorem_fiber_merge_tree_connected}, we show that~$\conf(X,\cellMTree)$, and hence~$\MT^{-1}(\cellMTree)$, is path connected when~$X$ has a branch point. We deduce a 1-1 correspondence between path connected components in the fiber~$\fib$ and non-isomorphic merge trees with barcode~$D$.
    \item We derive two important consequences of the above results for when~$X$ has at least one branch point: (i) in Corollary~\ref{corollary_counting_functions_in_the_fiber}, we find a lower-bound on the distance between the path connected components in~$\fib$, and (ii) in Corollary~\ref{corollary_distance_connected_comp}, we count the number of such components using existing work on merge trees~\cite{curry2018fiber, kanari2020trees}.
\end{itemize}

The paper is organised as follows. 

In Section \ref{sec:background} we formally define the notions of trees, geometric trees, and merge trees. Additionally, we define the notion of a cellular merge tree, a merge tree equipped with a suitable cellular structure. We also formally define persistent homology, describe the relationship between the local minima of a function and its zero dimensional barcode, and detail how the persistence map factors for functions on geometric trees.

In Section \ref{section_when_merge_trees_are_trees_and_metric}, we show that a function on a compact connected space has a cellular merge tree if and only if it has finitely many local minima. Then we define the interleaving distance between merge trees and show that it is a true metric on the subspace of cellular merge trees. Section \ref{section_functions_with_given_merge_tree} is devoted to providing necessary and sufficient conditions for when a function on a geometric tree~$X$ has a given cellular merge tree.

In Section \ref{section_characterizing_MT_inverse} we define the space~$\conf(X,\cellMTree)$ and a few other intermediary configuration spaces constrained by rules determined by~$\cellMTree$. By a series of consecutive homotopy equivalences between these configuration spaces, the section culminates in a proof of Theorem \ref{theorem_homotopy_equivalence_configuration_spaces}, showing that~$\MT^{-1}(\cellMTree)$, the space of continuous functions on~$X$ with merge tree~$\cellMTree$, is homotopy equivalent to~$\conf(X,\cellMTree)$. 

Section \ref{section_topological_consequences} exploits this homotopy equivalence to deduce topological properties of~$\MT^{-1}(\cellMTree)$ and $\fib$ for generic barcodes~$D$. The main result of this section, Theorem~\ref{theorem_fiber_merge_tree_connected}, says that~$\conf(X,\cellMTree)$ and hence~$\MT^{-1}(\cellMTree)$ are connected when~$X$ has at least one branch point, i.e.~$X$ is not homeomorphic to an interval. This then allows us to provide a lower bound on the distance between any two path connected components in~$\fib$ (Corollary~\ref{corollary_distance_connected_comp}), which depends only on the barcode~$D$. In addition, combining Theorem~\ref{theorem_fiber_merge_tree_connected} with existing work enumerating the number of merge trees with a given barcode \cite{curry2018fiber, kanari2020trees}, we deduce in Corollary \ref{corollary_counting_functions_in_the_fiber} that
\[\# \pi_0(\fib)=\prod_{[b,d)\in D} \# \big \{ [b',d') \in D \mid  [b,d) \subset [b',d')\big\}.\]
We conclude by computing the homotopy type of~$\PH^{-1}(D)$ via~$\conf(X,\cellMTree)$ when~$D$ has either one or two intervals. When~$D$ has one interval, we deduce that~$\PH^{-1}(D)$ is contractible (Corollary \ref{corollary_fiber_barcode_1_interval}). When~$D$ has two intervals, Corollary \ref{corollary_fiber_barcode_2_intervals} shows that~$\PH^{-1}(D)$ is homotopic to a wedge of
\begin{equation*}
    -1 + \sum_{v\in N(X)} (\eta(v) - 1)(\eta(v) - 2)
\end{equation*}
circles, where~$N(X)$ is the set of vertices in any triangulation of~$X$, and~$\eta(v)$ is the degree of vertex~$v$.

\section{Background}
\label{sec:background}
\subsection{Trees, merge trees and cellular merge trees}
\label{section_merge_trees}

A {\em tree} is a finite connected acyclic graph. It is {\em binary} if each vertex is the endpoint of at most~3 edges. The {\em geometric realization} of a tree~$\Tree$ is a topological space given by a copy of the interval~$[0,1]$ for each edge in~$\Tree$ with pairs of endpoints quotiented whenever they correspond to the same vertex of~$\Tree$. A {\em geometric tree} is the geometric realization of a tree.

Between any two points on a tree it is well known that there is a unique non self-intersecting path. We refer to this path as the \textit{shortest path} between a given two points. Indeed any other path connecting a given two points contains the shortest path in its image. When~$\Tree$ is geometric, the discussion extends to disjoint closed connected nonempty subsets~$A,B\subseteq \Tree$: there is a unique shortest path~$\ShortPath(A,B)$ connecting them. 

It follows that a subset~$S\subseteq \Tree$ of a geometric tree is path-connected if and only if it is connected. Namely, if~$S$ is not path connected, then the shortest path between~$a$ and~$b$ in~$S$ is not contained in~$S$. Taking~$U_1$ and~$U_2$ to be the connected components of~$T$ minus a point in this shortest path, but not in~$S$, we induce a disjoint open cover of~$S$.

We define the \textit{convex hull} of a collection~$\mathcal{C}$ of closed subsets, denoted~$\conv(\mathcal{C})$, as the union of points on shortest paths between elements of sets in~$C$. Clearly, convex hulls are connected.

A rooted tree~$(\Tree,r)$ is a tree~$\Tree$ with a distinguished vertex~$r$. A \textit{leaf} in a rooted tree is a vertex not equal to~$r$ adjacent to exactly one other vertex. A \textit{branch point} is a vertex adjacent to three or more vertices. If the root~$r$ is adjacent to two or more vertices, then we say that~$r$ is a branch point as well. A choice of root induces an orientation on the edges of any tree~$\Tree$ by the following procedure. We start by directing edges of~$\Tree$ adjacent to~$r$ away from~$r$. Inductively, if an edge between~$v$ and~$v'$ has not yet been oriented but an edge incident to~$v$ has been oriented, we orient the edge between~$v$ and~$v'$ from~$v$ to~$v'$. Whenever there is a directed edge from~$v$ to~$v'$ we say that~$v'$ is a \textit{child} of~$v$. 

In a rooted tree say that a vertex~$v'$ is a \textit{descendant} of~$v$ if there is a directed path from~$v$ to~$v'$, where potentially~$v = v'$. For rooted trees, we denote by~$\LCA(v,v')$ the least common ancestor of~$v$ and~$v'$.

Next, we introduce merge trees. We will use two distinct definitions of merge trees from the literature, which are both instances of~{\em gauged spaces}:

\begin{definition}
\label{definition_gauged_space}
A \textbf{gauged space} is a topological space~$X$ equipped with a continuous map~$\pi:X\rightarrow \R$. 

A \textbf{morphism} between two gauged spaces~$(X_1,\pi_1)$ and~$(X_2,\pi_2)$ is a continuous map~$\phi:X_1 \to X_2$ satisfying~$\pi_1 = \pi_2 \circ \phi$. An \textbf{isomorphism} of gauged space is a morphism that is also a homeomorphism.
\end{definition}

A continuous function~$f:X\rightarrow \R$ yields a merge tree as defined in~\cite{morozov2013interleaving}, which is an instance of gauged space:
\begin{definition}
For a topological space~$X$ with a continuous function~$f$, the associated \textbf{merge tree} $\MT(f)$ is the quotient of the space
\begin{equation*}
    \epi (f) := \{ (x,t)\in X\times \mathbb{R} :t \geq f(x) \}
\end{equation*}
by the relation~$(x,t) \sim (y,t)$ whenever~$x$ and~$y$ are in the same connected component of~$f^{-1}(-\infty,t]$.
\label{def:mergetree}
\end{definition}

Since merge trees inherit a map~$\pi_f$ from the second coordinate projection map on~$\epi(f)$, they are naturally viewed as gauged spaces. We illustrate the construction of a merge tree in Figure~\ref{fig:exmt}.

\begin{figure}[htbp]
\centering
\resizebox{\textwidth}{!}{
\includegraphics{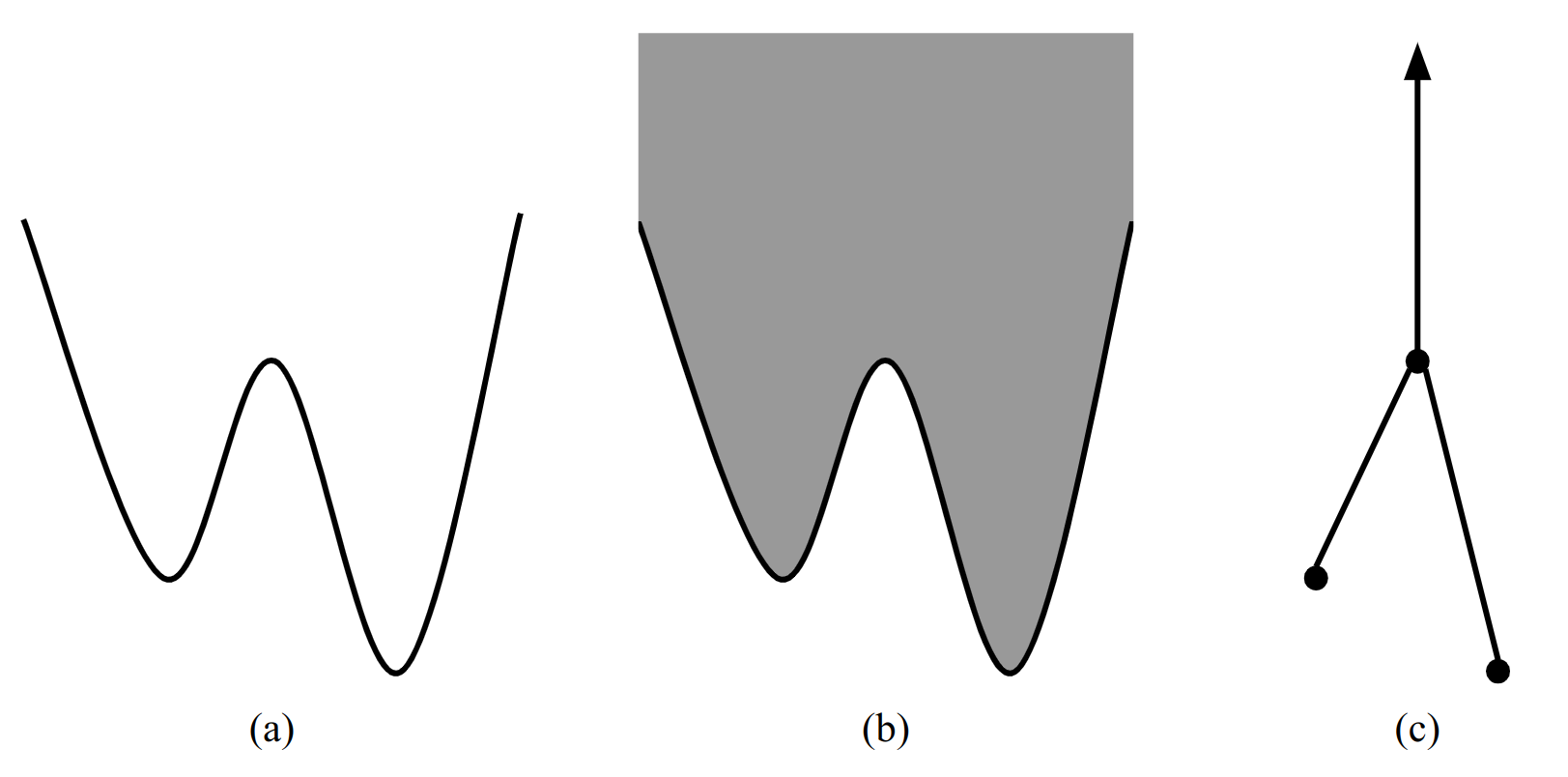}
}

    \caption{The construction of a merge tree from a function. (a) The graph of a function~$f$ on an interval. (b) Shaded is the content of~$\epi(f)$ strictly above the graph of~$f$. (c) By sending connected components of horizontal slices to points we obtain~$\MT(f)$. It happens that this merge tree has a cellular structure, although in general this may not be the case.}
    \label{fig:exmt}
\end{figure}

The other definition of merge trees that we will use appears for example in \cite{curry2018fiber}. We provide an analogous definition here, along with the notion of a labeling from~\cite{gasparovic2019intrinsic}.

\begin{definition}
\label{definition_cellular_merge_trees}
A \textbf{cellular merge tree}~$(\cellMTree, \pi)$ is the quotient space of a geometric rooted tree~$(\Tree',r)$ and a half open interval:
\begin{equation*}
    \Tree'\sqcup[0,1)/(r\sim 0),
\end{equation*}
equipped with a real-valued map~$\pi$ satisfying
\begin{itemize}
    \item~$\pi$ is strictly decreasing along edges in~$\Tree'$ oriented from the root~$r$.
    \item~$\pi$ is strictly increasing to infinity along the half open interval~$[0,1)$.
\end{itemize}
\end{definition}

The nodes of a merge tree~$\cellMTree$ are endowed with a partial order~$\preceq$ where~$\Labelnode \preceq \Labelnode'$ whenever~$\Labelnode$ is a descendent of~$\Labelnode'$.

Given a cellular merge tree~$(\cellMTree,f)$, the map from~$x \in f^{-1}(t)$ to the connected component of $f^{-1}(-\infty,t]$ containing~$x$ is a bijection. Further, if there is a path from~$x$ to~$y$ in~$\cellMTree$ along which~$f$ is increasing, then for~$t\leq f(y)$, the connected component containing~$x$ in~$f^{-1}(-\infty,t]$ is contained in the connected component containing~$y$ in~$f^{-1}(-\infty,f(y)]$. From this it follows that~$\cellMTree = \MT(f)$ and~$f = \pi_f$. Therefore, through the continuous injection~$(\cellMTree,f)\mapsto (\MT(f),\pi_f)$, cellular merge trees form a subspace of regular merge trees.

For notational convenience we will sometimes refer to the subset~$(0,1)$ of the half open interval of a cellular merge tree as~$e_\infty$. For indexing convenience, we will often work with {\em labelled} cellular merge trees, for which an arbitrary ordering of the leaves~$\Labelleaf_1,\cdots, \Labelleaf_n$ and of the nodes~$\Labelnode_1,\cdots,\Labelnode_m$ has been fixed. Although by default we assume that these orderings do not allow repetitions of the leaves and of the nodes, we will sometimes explicitly allow repetitions to make use of more general definitions and results from~\cite{gasparovic2019intrinsic}. For instance, the following definition allows labels with repetitions:
\begin{definition}
\label{def_matrix_merge_tree}
The \textbf{induced matrix} of a labelled cellular merge tree~$(\cellMTree,\pi)$ is given by
\begin{equation*}
   \LCAmatrix(\cellMTree)_{ij} := \pi(\LCA(\Labelleaf_i,\Labelleaf_j)).
\end{equation*}
\end{definition}

To simplify notations, when the context leaves is clear, we will write~$\MT(f)$ given a function~$f$ to designate the gauged space~$(\MT(f),\pi_f)$, and similarly~$\cellMTree$ to designate the cellular merge tree~$(\cellMTree,\pi)$.

\subsection{Persistent homology}
Fix a topological space~$X$ and a continuous function~$f:X\to \mathbb{R}$. The function~$f$ gives rise to a sequence of topological spaces~$f^{-1}(-\infty,t]$, nested by inclusion maps. Applying~$i^\mathrm{th}$ homology over a field~$\mathbb{F}$ to the sequence of spaces induces a sequence of vector spaces indexed by~$\R$. This sequence~$\persmod_i(f)$ of vector spaces is called the~$i^\mathrm{th}$ \textit{persistence module}~$f$. The persistence modules~$f$ can also be thought of as functors from~$(\R, \leq)$ to the category of vector spaces. If~$\persmod_i(f)$ is {\em pointwise finite dimensional} ({\em pfd} for short), i.e.~$\dim H_i(f^{-1}(-\infty,t])<\infty$ for all~$t$, then the~$i^{\text{th}}$ persistence module decomposes into a direct sum of modules \cite{crawley2015decomposition} indexed by a multiset~$D$,
\begin{equation}
\label{eq_decomposition_theorem}
   \persmod_i(f) \cong  \bigoplus_{I\in D} M_I,
\end{equation}
where each~$I\subseteq \R$ is an interval of the real line, and~$M_I$ is defined to be the sequence of vector spaces
\begin{equation*}
    M_I(t) = \begin{cases}
            \mathbb{F} & t\in I \\
            0 & \textrm{else,}
        \end{cases}
\end{equation*}
with associated maps
\begin{equation*}
    M_I(s,t) = \begin{cases}
            id & s,t\in I \\
            0 & \textrm{else.} 
        \end{cases}
\end{equation*}
The sequences of vector spaces~$M_I$ are called interval modules.

The multiset of intervals~$D$ is called the \textit{barcode} in dimension~$i$ of~$f$. We say that~$D$ is {\em finite} if it is a finite collection of intervals. We say that a function~$f$ is {\em pfd} if all its persistence modules~$\persmod_i(f)$ are pfd themselves. In this case, the collection of barcodes~$\{\PH_i(f)\}_{i\geq 0}$ associated to a function~$f$, abbreviated~$\PH(f)$, is well-defined and referred to as its {\em persistent homology}.

\subsection{Persistent homology and local minima}
\label{sec:persistent_homology_tree}

In this section we show some relations between the zero dimensional barcode~$\PH_0(f)$ and the number of local minima of a function~$f$. 

\begin{definition}
Given a topological space~$X$ and a map~$f:X\to\mathbb{R}$, a subset~$M\subseteq X$ is a \textbf{local minimum of~$f$} if~$M$ is connected,~$f$ is constant on~$M$, and any connected~$M'$ containing~$M$ also contains a point~$x$ satisfying~$f(x) > f(M)$.
\end{definition}

Note that if~$f$ is continuous then its local minima are each closed. Hence, if~$X$ is also compact, then the local minima of~$f$ are compact. In particular the local minima are compact when~$f$ is continuous and~$X$ is a geometric tree.

\begin{lemma}
\label{lemma_finitely_many_local_minima_weaker_version}
Let~$X$ be any topological space,~$D$ be a finite barcode, and~$f:X\to \mathbb{R}$ with~$\PH_0(f) = D$. Then~$f$ has finitely many local minima.
\end{lemma}

\begin{proof}
Let~$\LocalMin$ be a local minimum of~$f$, and~$m=f(\LocalMin)$. Assume, seeking contradiction, that no interval of~$D$ starts at~$m$. Then we can find a range~$[m-\epsilon, m]$ where no interval of~$D$ starts. Using the decomposition~\eqref{eq_decomposition_theorem} we see that the internal morphism~$\persmod_0(f)(m-\epsilon)\rightarrow \persmod_0(f)(m)$ is surjective. 

Note that~$\LocalMin\subseteq f^{-1}(-\infty,m]$ is connected in~$X$ and hence is a disjoint union of path connected subspaces. Pick one of these subspaces~$\LocalMin_0$. Therefore there is a path-connected component~$\ConnectedComp_{m-\epsilon}$ of~$f^{-1}(-\infty,m-\epsilon]$ which lies in the same path-connected component as~$\LocalMin_0$ in~$f^{-1}(-\infty,m]$. Given a path~$\gamma$ from~$\ConnectedComp_{m-\epsilon}$ to~$\LocalMin_0$ in~$f^{-1}(-\infty,m]$, the set~$\LocalMin'=\LocalMin\cup \im \gamma$ contradicts that~$\LocalMin$ is a local minimum.  

The same reasoning, working locally around local minima, shows that there are at least as many intervals in~$D$ starting at~$m$ as there are local minima with value~$m$. Since~$D$ is finite, this implies that~$f$ has finitely many local minima.
\end{proof}

The following result also ensures that the barcode of a continuous function on a tree with finitely many local minima is well-defined. 

\begin{lemma}
\label{lemma_finitely_many_local_minima_implies_pfd}
Let~$X$ be a tree and~$f:\tree\rightarrow \R$ be a continuous function with finitely many local minima. Then~$f$ is pfd. 
\end{lemma}

\begin{proof}
    Let~$t\in \R$. Let~$\ConnectedComp$ be a path-connected component of~$f^{-1}(-\infty,t]$. Consider the following alternative:
    \begin{enumerate}
    \item Either~$f$ has constant value~$t$ over~$\ConnectedComp$. Then~$M:=\ConnectedComp$ is a local minimum of~$f$, because any connected strict superset~$M'$ will be included in some~$f^{-1}(-\infty,t']$ with~$t'>t$, but not in~$f^{-1}(-\infty,t]$.
    \item Or~$f$ attains a minimum~$t'<t$ over one maximal connected subset~$M\subseteq \ConnectedComp$. Then~$M$ is also a local minimum of~$f$ in the whole~$\tree$. 
    \end{enumerate}
    In both cases, we can find a local minimum of~$f$ inside~$\ConnectedComp$, and since~$f$ has finitely many local minima, we deduce that~$f^{-1}(-\infty,t]$ has finitely many path-connected components, i.e.~$\dim H_0(f^{-1}(-\infty,t])<+\infty$. This is because a subset of a geometric tree is connected if and only if it is path connected. 
    
    Finally, since~$X$ is a tree, its subsets are component-wise contractible, hence~$\dim H_i(f^{-1}(-\infty,t])=0$ for all~$i>0$ and the result follows.
\end{proof}

\subsection{The fiber of persistent homology on a tree}
\label{sec:fromMTtoPH}
In this section we assume that~$\tree$ is a tree. Then all of its subsets are component-wise contractible, so its~$i^\mathrm{th}$ persistence modules are trivial for all~$i>0$. Hence, we refer to the zero dimensional barcode of~$f$ simply as~$\PH(f)$, the persistent homology of~$f$.

In this paper, we study the space of pfd continuous functions~$f:\tree\rightarrow \R$ giving rise to a fixed barcode~$D$:
\[\fib := \bigg\{ f:\tree\rightarrow \R \text{ pfd continuous} \mid \PH(f)=D\bigg\}.\]
We consider the topology on~$\fib$ induced by the supremum norm on continuous functions.

\begin{remark}
\label{remark_extend_fiber_problem_to_all_continuous_function}
We can also consider this inverse problem more generally in the space of all continuous functions by working directly at the level of persistence modules: studying the space of continuous functions~$f:\tree\rightarrow \R$ satisfying~$ \persmod_0(f) \cong  \bigoplus_{I\in D} M_I$. Our analysis could be conducted in this setting without substantial modifications, in particular because we will assume~$D$ to be finite. But to keep the exposition simple, this work assumes functions are pfd so that their barcodes are always defined.
\end{remark}

As observed in \cite{morozov2013interleaving}, the zero dimensional persistent homology of~$(X,f)$, for~$f:X\rightarrow \R$ a pfd function, is also the zero dimensional persistent homology of~$(\MT(f),\pi_f)$. Indeed, the dimension of~$H_0(f^{-1}(-\infty,t])$ is exactly the number of path components of~$f^{-1}(-\infty,t]$, which, being a subset of a tree, is the number of connected components of the same set. This is exactly~$|\pi_f^{-1}(t)|$. However, the set~$\pi_f^{-1}(-\infty,t]$ retracts onto~$\pi_f^{-1}(t)$ via the homotopy
\begin{equation*}
    h_u: (x,s) \longmapsto (x,s(1-u) + tu).
\end{equation*}
Hence for each~$t$ the map~$x \mapsto (x,t)$ induces pointwise isomorphisms between~$H_0(f^{-1}(-\infty,t])$ and $H_0(\pi_f^{-1}(-\infty,t])$. Further, these isomorphisms commute with inclusions arising from inequalities~$s \leq t$. Hence the persistence modules of~$(X,f)$ are completely determined by~$\MT(f)$. In other words the map~$\PH$ factors as a composite of maps
$$\xymatrix@1{\PH:  f \ar@{|->}[rr]^-{\MT} && \MT(f) \ar@{|->}[rr]&& D. }$$

We thus refer to the {\em persistent homology} on merge trees as the second of these maps, and its value on a given merge tree~$\MT(f)$ as the {\em barcode} of~$\MT(f)$. 

These observations naturally organise the problem of computing the fiber~$\fib$ into two consecutive steps: we will first study which functions have a given merge tree, and second, which merge trees have a given barcode. 

\section{Tree structure and metric for merge trees}
\label{section_when_merge_trees_are_trees_and_metric}
\subsection{When Merge trees are trees}
\label{subsection_when_merge_trees_are_trees}
It is tempting to assume that~$(\MT(f),\pi_f)$ is always a cellular merge tree, however this is not the case, even when~$X$ is very simple.

\begin{example}
\label{example_merge_tree_not_necessarily_cellular}
 If~$X = (-\infty, 0]$ and~$f(x) = e^x$,~$\MT(f)$ is an interval with two open endpoints. This cannot be a cellular merge tree since cellular merge trees have at most one open endpoint.
\end{example}

The following result gives conditions under which~$\MT(f)$ is indeed a cellular merge tree.

\begin{theorem}
\label{theorem_merge_tree_are_cellulars}
Let~$X$ be a compact connected space and~$f:X \to \R$ be a continuous function. Then~$(\MT(f),\pi_f)$ is a cellular merge tree if and only if~$f$ has finitely many local minima. More precisely,~$(\MT(f),\pi_f)$ is isomorphic to the labelled cellular merge tree~$(\cellMTree, \pi)$ with leaves~$\Labelleaf_1,\ldots,\Labelleaf_n$ if and only if the following conditions on~$f$ are satisfied:
\begin{enumerate}
    \item The function~$f$ has finitely many local minima~$X_1,\cdots,X_n$, with values~$\pi(\Labelleaf_1),\cdots ,\pi(\Labelleaf_n)$.
    \item For any~$1\leq i<j \leq n$, let~$t_{ij}$ denote the infimum of values~$t$ where~$X_i$ and~$X_j$ are in the same connected component of~$f^{-1}(-\infty,t]$, i.e.~$t_{ij}=\inf\{t \mid (X_i,t)\sim (X_j,t)\}$. Then
    \[t_{ij}=   \LCAmatrix(\cellMTree)_{ij}.\]
\end{enumerate}
\label{prop:mtcell}
\end{theorem}
We fix the function~$f:X \to \R$ and the cellular merge tree~$(\cellMTree, \pi)$. We will use two lemmas.

\begin{lemma}
\label{lem:minrep}
Let~$(y, t) \in\MT(f)$. There exists a local minimum~$\LocalMin$ of~$f$ such that~$(x,t)$ is also a representative of~$(y,t)$ for any~$x$ in~$\LocalMin$. In particular,~$\LocalMin$ can be chosen such that~$f(M) \leq f(y)$. As a result, for any~$t$, there are at most as many connected components of~$f^{-1}(-\infty,t]$ as local minima of~$f$.
\end{lemma}

\begin{proof}
\newcommand{\ConnectedCompY}{\Omega_y}

Fix~$y$ and~$t$. Let~$\ConnectedCompY$ denote the connected component of~$y$ in~$f^{-1}(-\infty,t]$. The set~$\ConnectedCompY$ is a closed subset of a compact set, and therefore is compact. Let~$m:=\min f_{|\ConnectedCompY}$, and let~$\LocalMin\subseteq\ConnectedCompY$ be a connected component of~$f^{-1}(m)\cap \ConnectedCompY$. We claim~$\LocalMin$ is a local minimum of~$f$ in~$X$.

Suppose~$\LocalMin'\subseteq X$ is a connected set containing~$\LocalMin$ on which~$f$ is never greater than~$m$. Since~$\LocalMin'\subseteq f^{-1}(-\infty,m] \subseteq f^{-1}(-\infty,t]$, we have~$\LocalMin'\subseteq\ConnectedCompY$, and therefore~$\LocalMin'\subseteq f^{-1}(m)$  because~$m=\min f_{|\ConnectedCompY}$. Thus~$M=M'$. It follows that~$\LocalMin$ is a local minimum of~$f$ in~$X$. Meanwhile, since~$\LocalMin$ minimizes~$f$ on~$\ConnectedCompY$, it must be the case that~$f(\LocalMin) \leq f(y)$.

To each connected component~$\ConnectedCompY\subseteq f^{-1}(-\infty,t]$, we associate one of its local minima~$\LocalMin\subseteq \ConnectedCompY$, and the last part of the lemma follows.
\end{proof}

\begin{lemma}
\label{lemma_infimum_same_components_is_minimum}
For any~$1\leq i<j \leq n$, local minima~$X_i$ and~$X_j$ are connected in~$f^{-1}(-\infty,t_{ij}]$.
\end{lemma}

\newcommand{\ConnectedCompXi}{\Omega_i}
\newcommand{\OutsideConnectedCompXi}{{\Omega}_j}
\newcommand{\NbghCompXi}{U_i}
\newcommand{\NbghOutsideCompXi}{U_j}
\newcommand{\minimumOutside}{m}
\begin{proof}
Suppose the opposite. Let~$\ConnectedCompXi$ be the connected component of~$X_i$ in~$f^{-1}(-\infty,t_{ij}]$. By Lemma~\ref{lem:minrep} there are finitely many connected components in~$f^{-1}(-\infty,t_{ij}]$, so~$\ConnectedCompXi$ is both open and closed. Thus~$\ConnectedCompXi$ and~$\OutsideConnectedCompXi:=f^{-1}(-\infty,t_{ij}]-\ConnectedCompXi$ are disjoint sets in~$X$ that are both open and closed. In particular,~$\OutsideConnectedCompXi$ is nonempty as it contains~$X_j$.  Since~$\ConnectedCompXi$ and~$\OutsideConnectedCompXi$ are each open,~$X - \ConnectedCompXi\cup \OutsideConnectedCompXi$ is closed and thus compact in~$X$. If~$X-\ConnectedCompXi\cup \OutsideConnectedCompXi$ is empty, this would imply~$X = \ConnectedCompXi\cup \OutsideConnectedCompXi$ is disconnected, contradicting the hypotheses of the theorem. Hence~$X-\ConnectedCompXi\cup \OutsideConnectedCompXi$ is nonempty and we may let~$\minimumOutside:=\min f_{|X - \ConnectedCompXi\cup \OutsideConnectedCompXi}$. 

The value~$\minimumOutside$ is greater than~$t_{ij}$ since~$f^{-1}(-\infty,t_{ij}]$ is contained in~$\ConnectedCompXi\cup \OutsideConnectedCompXi$. Therefore~$X=f^{-1}(-\infty,t_{ij}] \cup f^{-1}[m,+\infty)$ is disconnected, contradicting the hypotheses of the theorem.
\end{proof}

We next turn to the heart of the proof and find the actual tree structure of the merge tree~$\MT(f)$.

\newcommand{\GlobalMax}{A}
\newcommand{\ConnectedComponent}{\Omega}

\begin{proof}[Proof of Theorem \ref{prop:mtcell}]

Let~$x_1,\ldots,x_n$ be points in each of the local minima of~$X$ achieving the values~$m_1,\cdots,m_n$ under~$f$, and let:
\[\cellMTree:=  \bigg( \bigcup_{i=1}^n \{x_i\}\times [m_i,+\infty) \bigg)/ \bigg\{ (x_i,t)\sim (x_j,t) \mid t\geq t_{ij} \bigg\}. \]
So~$\cellMTree$ is a disjoint union of the~$n$ right-open intervals~$\{x_i\}\times [m_i,+\infty)$, and the~$i$-th interval is identified with the~$j$-th one at and beyond the threshold~$t_{ij}$. In particular,~$\cellMTree$ is a tree. 

We have a continuous map~$(x_i,t)\in \cellMTree \mapsto (x_i,t)\in \MT(f)$ which is well-defined by Lemma~\ref{lemma_infimum_same_components_is_minimum}. We also have a continuous map in the other direction, namely~$(x,t)\mapsto (x_i,t)$ where~$x_i$ is provided by Lemma~\ref{lem:minrep} to ensure~$(x,t)\sim (x_i,t)$ in~$\epi(f)$. Therefore~$\cellMTree$ and~$\MT(f)$ are isomorphic.
\end{proof}

\begin{remark}
If we defined merge trees using the equivalence relation~$(x,t) \sim (y,t)$ whenever~$x$ and~$y$ are in the same path component of~$f^{-1}(-\infty,t]$ instead of the same connected component, Theorem~\ref{prop:mtcell} would not hold. For a counterexample, consider the so-called topologist's sine curve
\begin{equation*}
    S = \{(0,t)\in \mathbb{R}^2 : t\in[-1,1]\}\cup\{(x,\sin\frac{1}{x})\in\mathbb{R}^2:x\in(0,1]\}.
\end{equation*}
It is well known that~$S$ is closed and connected but has two path components. Taking~$B$ to be a closed disk covering~$S$ and~$f:B\to\mathbb{R}$ the Euclidean distance from~$S$, we see that~$f^{-1}(-\infty,0]=S$ is not path-connected. Thus~$\MT(f)$ cannot be a cellular merge tree as it is not Hausdorff: any neighborhood of either point~$p$ with~$\pi_f(p)=0$ contains both such points.
\end{remark}

\begin{remark}
\label{remark_descendant_merge_tree}
In the conditions of Theorem~\ref{theorem_merge_tree_are_cellulars}, implying that~$\MT(f)$ is a cellular merge tree, a node~$\Labelnode$ is a descendant of a node~$\Labelnode'$ if and only if, when viewed as connected components of~$f^{-1}(-\infty, \pi(\Labelnode)]$ and~$f^{-1}(-\infty, \pi(\Labelnode')]$ respectively,~$\Labelnode$ is a subset of~$\Labelnode'$.
\end{remark}

\subsection{Interleaving distance on cellular merge trees}
\label{sec:cellular_interleaving}

In this section we analyze the interleaving distance, a pseudo-distance on merge trees~\cite[Lemma 1]{morozov2013interleaving}. We show that it is in fact a genuine distance on the subspace of cellular merge trees.

Aside from the map~$\pi_f$, merge trees also come equipped with {\em~$\epsilon$-shift} maps, for~$\epsilon \geq 0$:
\[i_f^\epsilon:(x,t)\in \MT(f) \longmapsto (x,t+\epsilon )\in \MT(f). \,\]
We will often omit subscripts, writing~$i_f^\epsilon$ as~$i^\epsilon$.

\begin{definition}
\label{definition_interleaving}
Let~$f$ and~$g$ be two continuous functions on~$X$. An \textbf{$\epsilon$-interleaving} between~$\MT(f)$ and~$\MT(g)$ is a pair of continuous functions~$\alpha^\epsilon:\MT(f) \rightarrow \MT(g)$,~$\beta^\epsilon : \MT(g) \rightarrow \MT(f)$ satisfying the following equations
\begin{align*}
    &\beta^\epsilon \circ \alpha^\epsilon = i^{2\epsilon} \qquad \pi_g(\alpha^\epsilon(x)) = \pi_f(x) + \epsilon \\
    &\alpha^\epsilon \circ \beta^\epsilon = i^{2\epsilon} \qquad \pi_f(\beta^\epsilon(y)) = \pi_g(y) + \epsilon.
\end{align*}

The \textbf{interleaving distance}~$d_I(\MT(f),\MT(g))$ is defined as the infimum of values~$\epsilon$ such that $\MT(f)$ and~$\MT(g)$ are~$\epsilon$-interleaved.
\end{definition}

We illustrate an interleaving between two merge trees in Figure \ref{fig:interleaving}.

\begin{figure}[htbp]
\centering
\resizebox{.5\textwidth}{!}{
\includegraphics{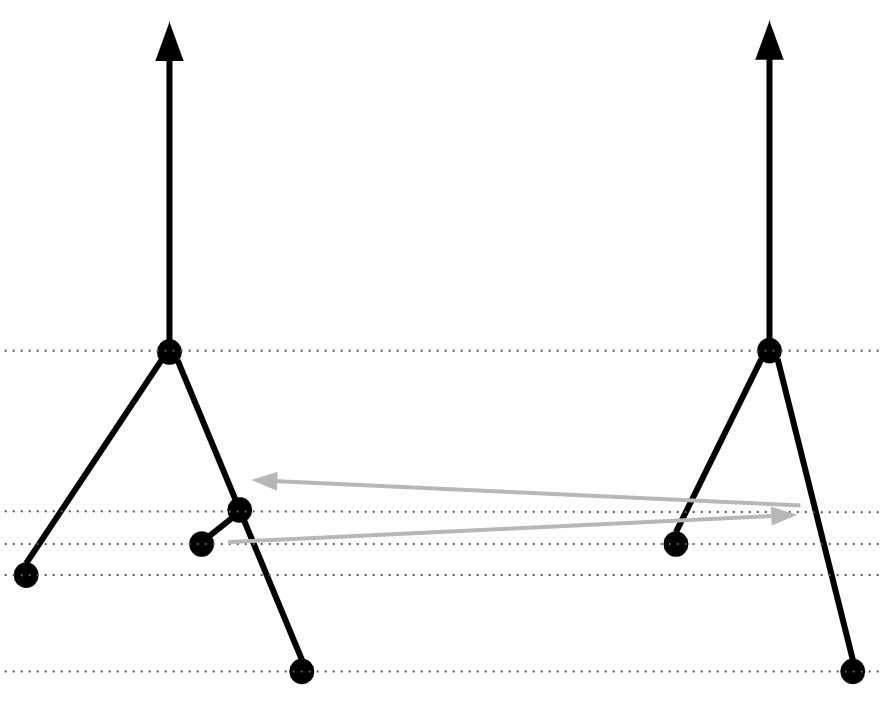}
}

    \caption{An interleaving between two merge trees.}
    \label{fig:interleaving}
\end{figure}
\begin{proposition}
\label{prop:intdist}
The interleaving distance is a metric on the subspace of cellular merge trees (up to isomorphism).
\end{proposition}

\begin{proof}
By~\cite[Lemma 1]{morozov2013interleaving} the interleaving distance is an extended pseudometric on cellular merge trees, so it remains to show that the interleaving distance is real valued on cellular merge trees and that two cellular merge trees~$(\cellMTree_1, \pi_1)$ and~$(\cellMTree_2,\pi_2)$ are isomorphic if they have interleaving distance zero.

By Corollary 4.4 of~\cite{gasparovic2019intrinsic} (which in turn depends upon \cite[Theorem 1]{touli2018fpt}), there exists two labellings  (possibly with repetitions) of~$\cellMTree_1$ and~$\cellMTree_2$ with the same number of indices of leaves, such that:
\begin{equation*}
    d_I(\cellMTree_1,\cellMTree_2) = \min_{1\leq i,j \leq N} |\LCAmatrix(\cellMTree_1)_{ij} - \LCAmatrix(\cellMTree_2)_{ij}|.
\end{equation*}
Therefore~$d_I$ is real-valued. In addition, by Lemma 2.9 of~\cite{gasparovic2019intrinsic}, if~$\LCAmatrix(\cellMTree_1)=\LCAmatrix(\cellMTree_2)$ then~$\cellMTree_1\cong \cellMTree_2$, and so~$d_I$ is a metric on cellular merge trees up to isomorphism.
\end{proof}

\begin{proposition}
\label{prop_small_distance_functions_implies_same_merge_trees}
Let~$X$ be a compact connected topological space, and let~$f,g:X\to \R$ be continuous pfd functions with the same finite barcode~$D$ in dimension zero. Let~$\delta_L$ be the minimum distance between pairs of non-equal interval left endpoints in~$D$ and~$\delta_R$ be the minimum distance between pairs of non-equal right endpoints. If
\begin{equation*}
    \|f-g\|_\infty < \min(\delta_L,\delta_R),
\end{equation*}
then~$\MT(f)$ and~$\MT(g)$ are isomorphic.
\end{proposition}

\begin{proof}
Suppose~$(\cellMTree,\pi)$ is a cellular merge tree with barcode~$D$. Then for any branch point~$v$ in~$\cellMTree$, the map~$H_0(\pi^{-1}(-\infty,\pi(v)-\epsilon]) \to H_0(\pi^{-1}(-\infty,\pi(v)])$ must have a nontrivial kernel for all sufficiently small~$\epsilon$. Therefore,~$\pi(v)$ must be the right endpoint of an interval of~$D$. Similarly, for any leaf~$l$ in~$\cellMTree$, the map~$H_0(\pi^{-1}(-\infty,\pi(l)-\epsilon]) \to H_0(\pi^{-1}(-\infty,\pi(l)])$ does not have full image for all sufficiently small~$\epsilon$, and so~$\pi(l)$ must be the left endpoint of an interval of~$D$.

By Lemma \ref{lemma_finitely_many_local_minima_weaker_version} and Theorem \ref{theorem_merge_tree_are_cellulars},~$\MT(f)$ and~$\MT(g)$ are both cellular. By Corollary 4.4 of~\cite{gasparovic2019intrinsic} (which in turn depends upon \cite[Theorem 1]{touli2018fpt}), there exist two labellings  (possibly with repetitions) of the cellular merge trees~$\MT(f)$ and~$\MT(g)$ with the same number of indices of leaves, such that:
\begin{equation*}
    d_I(\MT(f),\MT(g)) = \min_{1\leq i,j \leq N} |\LCAmatrix(\MT(f))_{ij} - \LCAmatrix(\MT(g))_{ij}|.
\end{equation*}
 The diagonal entries of~$\LCAmatrix(\MT(f))$ are projected values of leaves of~$\MT(f)$, and non-diagonal entries are projected values of branch points. Similarly for~$\LCAmatrix(\MT(g))$. However, from the first paragraph, the projected values of leaves and branch points of~$\MT(f)$ and~$\MT(g)$ are the values of left and right interval endpoints of~$D$ respectively. Thus if we assume~$d_I(\MT(f),\MT(g))$ is positive, then in fact, by our above equation,
\begin{equation*}
    d_I(\MT(f),\MT(g)) \geq \min(\delta_L,\delta_R).
\end{equation*}
Hence the stability theorem for the interleaving distance \cite[Theorem 2]{morozov2013interleaving} gives us that
\begin{equation*}
    \min(\delta_L,\delta_R) \leq d_I(\MT(f),\MT(g)) \leq \|f-g\|_\infty < \min(\delta_L,\delta_R),
\end{equation*}
a contradiction. So~$d_I(\MT(f),\MT(g)) = 0$. Since~$\MT(f)$ and~$\MT(g)$ are both cellular, Proposition~\ref{prop:intdist} implies that~$\MT(f)$ and~$\MT(g)$ are isomorphic.
\end{proof}

\section{Functions on a tree with a given merge tree}
\label{section_functions_with_given_merge_tree}
In this section we work with cellular merge trees whose branch points are non-degenerate:
\begin{definition}
\label{definition_non_degenerate_merge_tree}
A cellular merge tree~$(\cellMTree,\pi)$ is {\em generic} if it is a binary tree (each internal node has two children), and all leaves have distinct projection values.
\end{definition}

In particular if~$(\cellMTree,\pi)$ has~$n$ leaves then it has~$(n-1)$ internal nodes. For the rest of the section we fix a generic labelled cellular merge tree~$\cellMTree$, and for each internal node~$\Labelnode\in \cellMTree$, we fix an arbitrary labelling~$\leftChild \Labelnode$ and~$\rightChild \Labelnode$ of its two children. In particular~$\leftChild\Labelnode \preceq \Labelnode$, where as a reminder this notation means that~$\leftChild \Labelnode$ is a descendant of~$\Labelnode$, and likewise~$\rightChild\Labelnode \preceq \Labelnode$. For~$1\leq i \leq n$ we denote by~$m_i:=\pi(\Labelleaf_i)$ the value of a leaf. Let~$\tree$ be a geometric tree. 

\begin{proposition}
\label{proposition_functions_with_given_merge_tree_on_tree}
Let~$f:X\rightarrow \R$ be a continuous function with finitely many local minima. Then $\MT(f)$ is isomorphic to~$\cellMTree$ if and only if both the following conditions are satisfied:
\begin{enumerate}
    \item The function~$f$ has~$n$ local minima~$X_1,\cdots,X_n$ with values~$m_1,\cdots,m_n$.
    \item For any~$i\neq j$, the maximum of~$f$ restricted on~$\ShortPath(X_i,X_j)$ equals~$\LCAmatrix(\cellMTree)_{ij}$.
\end{enumerate}
\end{proposition}

\begin{proof}
By Theorem~\ref{theorem_merge_tree_are_cellulars},~$\MT(f)$ is a cellular merge tree, and it is isomorphic to~$\cellMTree$ if and only if:
\begin{enumerate}
    \item The function~$f$ has~$n$ local minima~$X_1,\cdots,X_n$ with value~$m_1,\cdots,m_n$.
    \item For any~$1\leq i<j \leq n$, denoting~$t_{ij}=\inf\{t \mid (X_i,t)\sim (X_j,t)\}$, we have~$t_{ij}=   \LCAmatrix(\cellMTree)_{ij}$.
    \end{enumerate}
Furthermore by~Lemma~\ref{lemma_infimum_same_components_is_minimum} we can replace the infimum by a minimum in the definition of~$t_{ij}$. Since~$X$ is a geometric tree it is then clear that
\[\min\{t\mid X_i \text{ and } X_j \text{ are connected in } f^{-1}(-\infty,t]\}= \max f_{|\ShortPath(X_i,X_j) } \qedhere\]
\end{proof}

We arrive at our most useful characterisation of functions with a given merge tree. 

\begin{proposition}
\label{proposition_functions_with_given_merge_tree_on_tree_convex_hulls}
Let~$f:X\rightarrow \R$ be a continuous function with finitely many local minima. Then $\MT(f)$ is isomorphic to~$\cellMTree$ if and only if both the following conditions are satisfied:
\begin{enumerate}
    \item The function~$f$ has~$n$ local minima~$X_1,\cdots,X_n$ with values~$m_1,\cdots,m_n$.
    \item Given a node~$\Labelnode$ (possibly a leaf), let $X_f = (X_1,\ldots,X_m)$ and
\[\conv_{X_f}(\Labelnode):=\conv\bigg\{X_i \mid \text{ leaf } \Labelleaf_i \text{ is a descendent of } \Labelnode \text{ in }\MT(f) \bigg\}\subseteq X.\]
Then, for any~$1\leq k \leq n-1$, we have:
\begin{equation}
    \label{eq_max_f_shortest_path_conv_hulls}
    \max \big\{f(x) \mid x \in \ShortPath(\conv_{X_f}(\leftChild\Labelnode_k),\conv_{X_f}(\rightChild\Labelnode_k))\big\}=\pi(\Labelnode_k),
\end{equation}
and the maximum is attained at a unique connected subset~$Y_k$ of the shortest path.
\end{enumerate}
\end{proposition}

\begin{proof}
Let~$f$ satisfy the conditions of the statement. One is condition~1 of Proposition~\ref{proposition_functions_with_given_merge_tree_on_tree}, and an induction on nodes of~$\cellMTree$ in increasing order of~$\pi$-value immediately yields condition 2 as well, so~$\MT(f)\sim \cellMTree$. Conversely, condition 2 in Proposition~\ref{proposition_functions_with_given_merge_tree_on_tree} is equivalent to
\begin{equation*}
   \forall \Labelnode_k\in \cellMTree, \forall \Labelleaf_i\preceq \leftChild\Labelnode_k \text{ and } \Labelleaf_j \preceq \rightChild\Labelnode_k, \,  \max f_{|\ShortPath(X_i,X_j) }= \pi(\Labelnode_k).
\end{equation*}
From this, an immediate induction yields that, for each node~$\Labelnode_k\in \cellMTree$, the value~$\max f_{|\conv_{X_f}(\Labelnode_k)}$ equals~$\pi(\Labelnode_k)$ and is attained on~$\ShortPath(\conv_{X_f}(\leftChild\Labelnode_k), \conv_{X_f}(\rightChild\Labelnode_k))$.

Assume, seeking contradiction, that the maximum~$\pi(\Labelnode_k)$ is attained at two distinct connected subsets~$Y_1$ and~$Y_2$ of the shortest path. Then, inside~$f^{-1}(-\infty, t]$ for~$t<\pi(\Labelnode_k)$, the connected component of elements between~$Y_1$ and~$Y_2$ is distinct from that of elements of~$\conv_{X_f}(\leftChild\Labelnode_k)$ and~$\conv_{X_f}(\rightChild\Labelnode_k)$, and at~$t=\pi(\Labelnode_k)$ we thus have three or more connected components of the sublevel-sets of~$f$ which are identified, contradicting that~$\cellMTree$ is generic. 
\end{proof}

\section{Retraction of the fiber to configuration space on a tree}
\label{section_characterizing_MT_inverse}
Let~$\tree$ be a geometric tree. We metrize spaces of functions on~$X$ via the supremum norm. In this section we fix a generic labelled cellular merge tree~$\cellMTree$ and we analyze the subspace of functions~$f$ in the fiber:
\[\fibT= \bigg\{f:X\rightarrow \R, \, \MT(f)=\cellMTree \bigg\}.\]

We assume without loss of generality that~$\cellMTree$ has only branch points and leaves as nodes. We will simplify~$\fibT$ by means of a series of homotopy equivalences
\begin{equation*}
    \fibT \isomto \critconf(X,\cellMTree) \isomto \minconf(X,\cellMTree) \isomto \conf(X,\cellMTree),
\end{equation*}
where the spaces~$\critconf(X,\cellMTree)$,~$\minconf(X,\cellMTree)$, and~$\conf(X,\cellMTree)$ are configuration spaces tracking the local minima and saddles of a function~$f\in \fibT$, detailed hereafter.

Consider~$\tilde{X} = (X_1,\ldots,X_n) \subseteq X^n$. Motivated by the definition we made in Proposition \ref{proposition_functions_with_given_merge_tree_on_tree_convex_hulls}, for a node~$\Labelnode$ of~$\cellMTree$, possibly a leaf, we define:
\[\conv_{\tilde{X}}(\Labelnode):=\conv\Big\{X_i \mid \text{ leaf } \Labelleaf_i \text{ is a descendent of }\Labelnode \Big\}\subseteq X.\]
We define~$\conv_x(\Labelnode)$ for~$x = (x_1,\ldots,x_n)\in X^n$ similarly. We illustrate this construction in Figure~\ref{fig:convv}.

\begin{figure}[htbp]
\centering
\resizebox{.9\textwidth}{!}{
\includegraphics{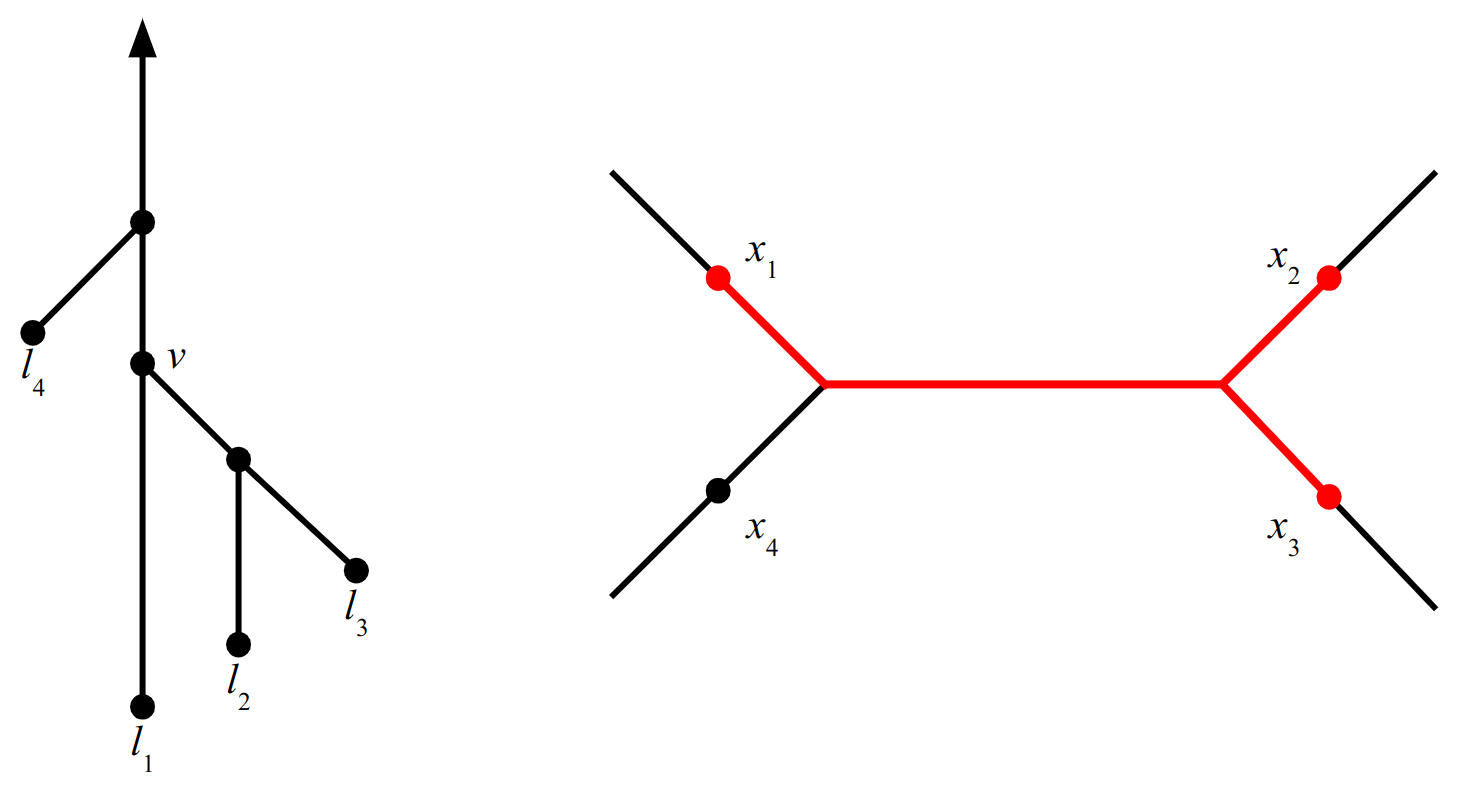}
}

    \caption{(Left) A cellular merge tree~$\cellMTree$ with four leaves. (Right) An element $x = (x_1,x_2,x_3,x_4)$ of~$\conf_4(X)$, for~$\tree$ a tree. Here, the set $\conv_x(\Labelnode)$ associated to a node~$\Labelnode\in \cellMTree$ is highlighted.}
    \label{fig:convv}
\end{figure}

Denote the usual ordered configuration space on~$n$ points by~$\conf_n(X)$. The space~$\conf(X,\cellMTree)$ is given by
\[\conf(X,\cellMTree):=\bigg\{x = (x_1,\cdots,x_n) \in \conf_n(X) \mid \conv_x( \Labelnode) \cap \conv_x( \Labelnode') \neq \emptyset \Rightarrow \Labelnode\preceq  \Labelnode' \text{ or }  \Labelnode'\preceq  \Labelnode \bigg\}.\]
A configuration~$(x_1,\cdots,x_n)\in \conf(X,\cellMTree)$ should be thought of as the points where a function~$f$ with merge tree~$\cellMTree$ achieves its local minima~$m_1,\cdots,m_n$. Because in general a function~$f$ with merge tree~$\cellMTree$ could achieve its minima on arbitrary closed sets rather than points, it is natural to extend~$\conf(X,\cellMTree)$ to the following configuration space of closed sets:
\begin{align*}
    \minconf(X,\cellMTree) := & \bigg\{\tilde{X}=(X_1,\cdots,X_n) \subseteq X^n \text{ disjoint connected closed sets}\mid\\ 
    &\conv_{\tilde{X}}( \Labelnode) \cap \conv_{\tilde{X}}( \Labelnode') \neq \emptyset \Rightarrow  \Labelnode\preceq  \Labelnode' \text{ or }  \Labelnode'\preceq  \Labelnode \bigg\}.
\end{align*}
Note that a set configuration~$\tilde{X}=(X_1,\cdots,X_n)\in \minconf(\tree,\cellMTree)$ induces convex hulls~$\conv_{\tilde{X}}(\Labelnode)$ for any node~$\Labelnode\in \cellMTree$. For~$\Labelnode$ an internal node with children~$\leftChild\Labelnode$ and~$\rightChild\Labelnode$, we will also consider the following subset of~$X$:
\[\ShortPath_{\tilde{X}}(v) := \ShortPath\big(\conv_{\tilde{X}}(\leftChild\Labelnode),\conv_{\tilde{X}}(\rightChild\Labelnode)\big). \]
We are now ready introduce our last configuration space~$\critconf(X,\cellMTree)$ of closed sets where we also record saddles of a function~$g\in \fibT$:
\begin{align*}
    \critconf(X,\cellMTree) := & \bigg\{(\tilde{X},\tilde{Y})=(X_1,\cdots,X_n,Y_1,\cdots, Y_{n-1}) \subseteq X^{2n-1} \text{ disjoint connected closed sets}\mid \\
    &(X_1,\cdots,X_n)\in \minconf(\tree,\cellMTree),\\
    & \forall j, \, Y_j \cong [0,1], \, Y_j \text{ subset of the interior of } \ShortPath_{\tilde{X}}(\Labelnode_j) \bigg\}.
\end{align*}
Note that the Hausdorff distance inherited from the ground metric on~$X$ induces topologies on configuration spaces of closed sets, hence subset topologies for our spaces~$\critconf(X,\cellMTree),\minconf(X,\cellMTree)$ and~$\conf(X,\cellMTree)$.

\begin{theorem}
\label{theorem_homotopy_equivalence_configuration_spaces}
The spaces~$\fibT, \critconf(X,\cellMTree), \minconf(X,\cellMTree)$, and~$\conf(X,\cellMTree)$ are homotopy equivalent.
\end{theorem}

The proof of Theorem~\ref{theorem_homotopy_equivalence_configuration_spaces} decomposes as the construction of three consecutive homotopy equivalences:
\begin{equation*}
    \fibT \isomto \critconf(X,\cellMTree) \isomto \minconf(X,\cellMTree) \isomto \conf(X,\cellMTree).
\end{equation*}
\begin{proof}
For the proof we will fix a metric~$d$ on~$X$.

\subsubsection*{Step 1:~$\minconf(X,\cellMTree)\isomto \conf(X,\cellMTree)$}

We choose an arbitrary leaf~$\tau\in X$ as the root of the tree, hence for any connected closed subset~$A\subseteq X$, the point~$\tau(A)$ which is closest to the root is uniquely defined. We can continuously contract~$A$ to~$\tau(A)$ with a family~$(A_t)_{0\leq t \leq 1}$ of enclosed subsets:
\[A_t:= \big\{ x\in A \mid d(a,\tau(A))\leq (1-t)\mathrm{diam}(A)\big\}.\]

Given a configuration of connected closed sets~$(X_1,\cdots, X_n)$, the continuous contraction of each~$X_i$ to~$\tau(X_i)$ defines a deformation retract of~$\minconf(\tree,\cellMTree)$ to~$\conf(\tree,\cellMTree)$:
\[\homotopy: \big(t,(X_1,\cdots,X_n)\big)\in [0,1]\times \minconf(X,\cellMTree)\longmapsto ((X_1)_t,\cdots,(X_n)_t)\in \minconf(X,\cellMTree)\]
Indeed, under this map, any~$\tilde{X}\in \minconf(X,\cellMTree)$ is mapped to an element~$\homotopy(t,\tilde{X})$ that satisfies the condition for being in~$\minconf(X,\cellMTree)$: for any nodes~$\Labelnode,\Labelnode'\in \cellMTree$, if~$\conv_{\homotopy(t,\tilde{X})}(\Labelnode)\cap \conv_{\homotopy(t,\tilde{X})}(\Labelnode')\neq \emptyset$, since~$\conv_{\homotopy(t,\tilde{X})}(\Labelnode)\subseteq \conv_{\tilde{X}}(\Labelnode)$ and~$\conv_{\homotopy(t,\tilde{X})}(\Labelnode')\subseteq \conv_{\tilde{X}}(\Labelnode')$, then we also have~$\conv_{\tilde{X}}(\Labelnode)\cap \conv_{\tilde{X}}(\Labelnode')\neq \emptyset$, and so~$\Labelnode\preceq \Labelnode'$ or~$\Labelnode'\preceq \Labelnode$, as desired.

\subsubsection*{Step 2:~$\critconf(X,\cellMTree) \isomto \minconf(X,\cellMTree)$}

We continue to use the root~$\tau\in X$. Let~$\tilde{X}=(X_1,\cdots,X_n)\in \minconf(X,\cellMTree)$ be a configuration of closed sets. Given~$1\leq j \leq (n-1)$, we have that~$\ShortPath_{\tilde{X}}(\Labelnode_j)$ is a closed segment~$[a_j,b_j]$ in~$X$, where by convention the extreme~$a_j$ is the closest to the root~$\tau$.

Let~$(s,t)$ be an element of the open standard simplex~$\Delta_2:=\{(s,t)\mid 0<s\leq t<1\}$. Define
\[Y_j^{s,t}:=\bigg\{ x\in [a_j,b_j] \mid  s \leq \frac{d(a_j,x)}{d(a_j,b_j)} \leq t   \bigg\}.\]
Note that, for varying~$1\leq j \leq (n-1)$, the shortest paths~$\ShortPath_{\tilde{X}}(\Labelnode_j)$ are disjoint from each other. Indeed, let's assume, seeking contradiction, that~$\ShortPath_{\tilde{X}}(\Labelnode)$ intersects~$\ShortPath_{\tilde{X}}(\Labelnode')$ for some distinct nodes~$\Labelnode,\Labelnode'\in \cellMTree$. We then have~$\conv_{\tilde{X}}(\Labelnode)\cap \conv_{\tilde{X}}(\Labelnode')\neq \emptyset$, and since~$\tilde{X}\in \minconf(X,\cellMTree)$, we can assume without loss of generality that~$\Labelnode$ is a descendant of~$\Labelnode'$. But then~$\conv_{\tilde{X}}(\Labelnode)$ is disjoint from~$\ShortPath_{\tilde{X}}(\Labelnode')$, and since~$\ShortPath_{\tilde{X}}(\Labelnode)\subseteq \conv_{\tilde{X}}(\Labelnode)$, we reach the contradiction $\ShortPath_{\tilde{X}}(\Labelnode)\cap \ShortPath_{\tilde{X}}(\Labelnode')=\emptyset$. Therefore the sets~$Y_j^{s_j,t_j}$ are disjoints from each other and from the sets~$X_i$, and we have the homeomorphism:
\[[(X_i)_{i=1}^n, (s_j,t_j)_{j=1}^{n-1}]\in \minconf(X,\cellMTree) \times (\Delta_2)^{n-1} \longmapsto [(X_i)_{i=1}^n, (Y_j^{s_j,t_j})_{j=1}^{n-1}] \in \critconf(X,\cellMTree).\]
The deformation retract of each copy of~$\Delta_2$ to a point gives us the homotopy equivalence from $\critconf(X,\cellMTree)$ to~$\minconf(X,\cellMTree)$.

\subsubsection*{Step 3:~$\fibT \isomto \critconf(X,\cellMTree)$}
To show that there is a homotopy equivalence between these two spaces, we will need to define a map which sends a function~$f$ to its local minima~$X_i(f)$ and saddles~$Y_j(f)$, see Figure 
\ref{fig:exsaddlemin}.

\begin{figure}[htbp]
\centering
\resizebox{.75\textwidth}{!}{
\includegraphics{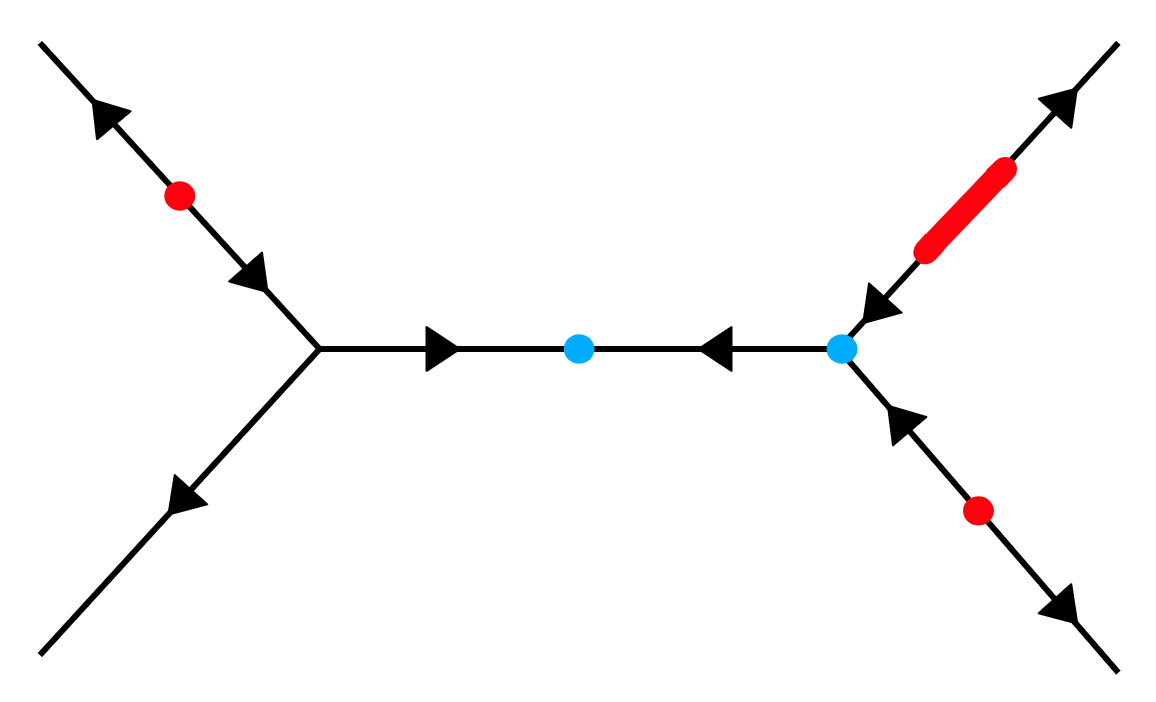}
}
    \caption{The local minima and saddles of a function on a tree. Directions of arrows on the depicted tree~$X$ indicate where the function is increasing. The three red regions indicate the locations of local minima~$X_i(f)$ while the two blue regions indicate the locations of saddles~$Y_j(f)$. Saddles do not need to be local maxima: in this example, one saddle is a local maximum while the other is not.}
    \label{fig:exsaddlemin}
\end{figure}

Let~$f\in \fibT$. For~$1\leq i \leq n$, let~$X_i(f)$ denote the connected subset of~$X$ where~$f$ achieves the minimum~$m_i$ corresponding to leaf~$\Labelleaf_i$ of~$\cellMTree$. The map~$f\mapsto X_i(f)$ is continuous.

Let~$\Labelnode,\Labelnode'\in \cellMTree$ be two nodes such that~$A:=\conv_{\tilde{X}}(\Labelnode)\cap \conv_{\tilde{X}}(\Labelnode')\neq \emptyset$. We assume without loss of generality that~$\pi(\Labelnode)\leq\pi(\Labelnode')$ and show that~$\Labelnode$ is a descendant of~$\Labelnode'$, i.e.~$\Labelnode\preceq\Labelnode'$. The node~$\Labelnode$, viewed as a connected component in~$f^{-1}(-\infty, \pi(\Labelnode)]$, contains~$A$. Since the connected component of~$f^{-1}(-\infty, \pi(\Labelnode')]$ represented by~$\Labelnode'$ also contains~$A$, viewing nodes once again as connected components of sublevel sets,~$\Labelnode$ is a subset of~$\Labelnode'$. So~$\Labelnode\preceq\Labelnode'$ (see Remark~\ref{remark_descendant_merge_tree}). Therefore~$\tilde{X}:= (X_1(f),\ldots,X_n(f))  \in \minconf(\tree,\cellMTree)$.

By Proposition~\ref{proposition_functions_with_given_merge_tree_on_tree_convex_hulls}, for~$1\leq j \leq n-1$, the restriction of~$f$ to~$\ShortPath_{\tilde{X}}(\Labelnode_j)$ attains its maximum~$\pi(\Labelnode_j)$ on a unique connected closed set~$Y_j(f)$. Since local minima~$X_i(f)$ vary continuously with~$f$, so do the convex hulls~$\conv_{\tilde{X}}(\Labelnode_j)$ and the shortest paths~$\ShortPath_{\tilde{X}}(\Labelnode_j)$ between them, and therefore the maps~$f\mapsto Y_j(f)$ are continuous, and so we have defined a continuous map:
\[F: f\in \fibT \longmapsto \big(X_1(f),\cdots,X_n(f),Y_1(f),\cdots,Y_{n-1}(f)\big) \in \critconf(X,\cellMTree).\]
To show that~$F$ is a homotopy equivalence, we define a map in the other direction. Let~$Z=(\tilde{X},\tilde{Y})=(X_1,\cdots,X_n,Y_1,\cdots,Y_{n-1})\in \critconf(X,\cellMTree)$. We construct a function~$f_Z$ 
 by induction on the nodes of~$\cellMTree$. To begin with, we define~$f_Z(X_i):= m_i$. Next, let~$\Labelnode_j\in \cellMTree$ be a node such that~$f_Z$ is already defined on~$\conv_{\tilde{X}}(\leftChild\Labelnode_j)$ and~$\conv_{\tilde{X}}(\rightChild\Labelnode_j)$. We extend~$f_Z$ on
\[\conv_{\tilde{X}}(\Labelnode_j)=\conv_{\tilde{X}}(\leftChild\Labelnode_j)\cup \conv_{\tilde{X}}(\rightChild\Labelnode_j)\cup \ShortPath_{\tilde{X}}(\Labelnode_j),\]
by letting~$f_Z(Y_j):=s_j$ and by linear interpolation on the rest of~$\ShortPath(\Labelnode_j)$. At the end of this process,~$f_Z$ is defined on the convex hull of all the~$X_i$, outside of which we let~$f_Z$ increase in all directions:
\[\forall x\in X\setminus \conv(X_1,\cdots,X_n),\qquad f_Z(x):= f_Z(\mathrm{proj}_{\conv(X_1,\cdots,X_n)}(x))+ d(x,\conv(X_1,\cdots,X_n)).\]
By Proposition~\ref{proposition_functions_with_given_merge_tree_on_tree_convex_hulls}~$\MT(f_Z)=T$ as desired.

This gives us a continuous map:
\[G: Z\in \critconf(X,\cellMTree) \longmapsto f_Z\in \fibT,\]
and clearly~$F\circ G= \mathrm{Id}$. Conversely, by Proposition~\ref{proposition_functions_with_given_merge_tree_on_tree_convex_hulls}, the straight-line interpolation
\[(t,f)\in [0,1]\times \fibT\longmapsto (1-t)f +t G\circ F(f)\]
is valued in~$\fibT$, hence it defines a homotopy equivalence~$G\circ F\sim \mathrm{Id}$.
\end{proof}

\section{The topology of functions on a tree with a given barcode}
\label{section_topological_consequences}
Let~$\tree$ be a geometric tree. In this section we fix a barcode~$D$ and analyze the space~$\fib$ of continuous pfd functions with barcode~$D$. For~$f\in \fib$ we let~$\fibf\subseteq \fib$ be the connected component of the fiber that contains~$f$. In the rest of the section we assume that~$D$ is generic in the following sense:

\begin{definition}
\label{definition_generic_barcode}
A barcode~$D$ is {\em generic} if it is finite and all its interval endpoints are distinct.
\end{definition}

Computing the fiber~$\fib$ can be solved by computing the fibers of two composite maps:
$$\xymatrix@1{\PH:  f \ar@{|->}[rr]^-{\MT} && \cellMTree \ar@{|->}[rr] && D }$$
The fiber of the second map is known from~\cite{curry2018fiber}: the number of merge trees giving rise to a generic barcode~$D$ is finite and computed computed in~\cite[Theorem~4.8]{curry2018fiber}. Therefore it remains to analyze functions with a given merge tree~$\cellMTree$. 

We first identify~$\fibT$ with a connected component of~$\fib$ in subsection~\ref{subsection_counting_connected_components}, and then, in subsection~\ref{subsection_topology_connected_components}, we derive some topological properties of these components. Before we begin we first record the following two useful properties.

\begin{proposition}
\label{prop_generic_barcode_implies_generic_merge_tree}
If~$D$ is a generic barcode and~$f\in \fib$, then the merge tree~$\MT(f)$ is a generic cellular merge tree.
\end{proposition}

\begin{proof}
Since~$D$ is finite, by Lemma~\ref{lemma_finitely_many_local_minima_weaker_version}~$f$ has finitely many local minima, and in turn~$(\MT(f),\pi_f)$ is cellular by Theorem~\ref{theorem_merge_tree_are_cellulars}. Since~$\PH(\pi_f)=\PH(f)$ has distinct interval endpoints, no more than two connected subsets can merge at a time in the sublevel-sets of~$\pi_f$, hence~$\MT(f)$ is a binary tree. Similarly, no two leaves of~$\MT(f)$ can have the same value of~$\pi_f$, as this would force~$\PH(f)$ to have a repeated left endpoint.
\end{proof}

\begin{proposition}
\label{prop_fiber_merge_tree_non_empty}
Given a geometric tree~$\tree\neq \emptyset$ and a cellular merge tree~$\cellMTree$, the fiber~$\fibT$ and the configuration space~$\conf(\tree,\cellMTree)$ are non-empty.
\end{proposition}

\begin{proof}
A function on the unit interval (hence more generally on any non-empty tree~$\tree$) with merge tree~$\cellMTree$ can be constructed e.g. as in~\cite[Proposition 6.8]{curry2018fiber}. Therefore~$\fibT\neq \emptyset$, and by Theorem~\ref{theorem_homotopy_equivalence_configuration_spaces},~$\conf(\tree,\cellMTree)\neq \emptyset$ as well.
\end{proof}

\subsection{Counting connected components in the fiber over binary trees}
\label{subsection_counting_connected_components}
Recall that the barcode~$D$ is generic. To avoid trivial cases where~$\fib=\emptyset$, we further assume that~$D$ has no interval in degree greater than~$0$, and only one unbounded interval~$[b,\infty)$ in degree~$0$ that contains all other intervals. 

\begin{proposition}
\label{prop_merge_tree_locally_constant}
Given~$\cellMTree$ a generic cellular merge tree with barcode~$D$,~$\fibT$ is a non-empty union of connected components in~$\fib$.
\end{proposition}

\begin{proof}
~$\fibT\neq \emptyset$ (Proposition~\ref{prop_fiber_merge_tree_non_empty}) and~$\MT$ is locally constant on~$\fib$ (Proposition~\ref{prop_small_distance_functions_implies_same_merge_trees}).\qedhere
\end{proof}

\begin{theorem}
\label{theorem_fiber_merge_tree_connected}
Let~$X$ be a tree not homeomorphic to the unit interval. Given a generic cellular merge tree~$\cellMTree$, the fiber~$\fibT$ is nonempty and path-connected. In particular, if~$f\in \fibT$ has barcode~$D$, then~$\fibT$ equals~$\fibf$, the path connected component of the fiber~$\fib$ containing~$f$.
\end{theorem}

The proof mainly relies on the following result.

\begin{proposition}
\label{proposition_conf_space_connected}
Let~$X$ be a tree not homeomorphic to the unit interval. Let~$\cellMTree$ be a generic cellular merge tree. Then~$\conf(X,\cellMTree)$ is path-connected. 
\end{proposition}

Before proving the proposition, let us see how it leads to the theorem.

\begin{proof}[Proof of Theorem~\ref{theorem_fiber_merge_tree_connected}]
From Proposition~\ref{prop_fiber_merge_tree_non_empty},~$\fibT\neq \emptyset$, and from Theorem~\ref{theorem_homotopy_equivalence_configuration_spaces}, it is homotopy equivalent to~$\conf(\tree,\cellMTree)$, which is path-connected by Proposition~\ref{proposition_conf_space_connected}. If~$\cellMTree$ has barcode~$D$, by Proposition~\ref{prop_merge_tree_locally_constant},~$\fibT$ is a non-empty union of connected components of~$\fib$. Therefore it equals exactly one such connected component.
\end{proof}

To prove that~$\conf(X,\cellMTree)$ is path-connected, we proceed in two steps, each relying on a key lemma. First, we show how to deform a configuration of points into one whose points all lie on a common edge of~$X$. The main result to achieve this step is Lemma~\ref{lem:gatherpoints}. Then, given two configurations whose points lie on an edge, we show how to connect them using a branch point of~$X$. We achieve this step with Lemma~\ref{lemma_connected_config_space_starlike_tree}. Throughout, we fix a labelling of the cellular merge tree~$\cellMTree$.

The following two results, Lemma~\ref{lem:moveonepoint} and Lemma~\ref{lem:moveoffbranch}, will be used repeatedly during our argument.

\begin{lemma}
\label{lem:moveonepoint}
Let~$x = (x_1, \ldots,x_n)\in\conf(X,\cellMTree)$, and~$x' = (x_1,\ldots,x_{i-1},y,x_{i+1},\ldots,x_n) \in \conf_n(X)$. Fix~$P$, the image of some path from~$x_i$ to~$y$. If~$\conv_x(\Labelnode)$ does not intersect~$P$ for any~$\Labelnode$ that is not an ancestor of~$\Labelleaf_i$, then~$x'\in\conf(X,\cellMTree)$ and there is a path from~$x$ to~$x'$ in~$\conf(X,\cellMTree)$.
\end{lemma}

\begin{proof}
For~$p\in P$, let~$x(p) = (x_1,\ldots, x_{i-1}, p, x_{i+1},\ldots,x_n)$. Let~$\Labelnode,\Labelnode'$ be nodes of~$T$ with neither~$\Labelnode\preceq \Labelnode'$ nor~$\Labelnode'\preceq \Labelnode$. Therefore it cannot be the case that both~$\Labelnode$ and~$\Labelnode'$ are ancestors of~$\Labelleaf_i$ as ancestors of~$\Labelleaf_i$ are totally ordered. If neither~$\Labelnode$ nor~$\Labelnode'$ are ancestors of~$\Labelleaf_i$ then
\begin{equation*}
    \conv_{x(p)}(\Labelnode)\cap \conv_{x(p)}(\Labelnode') = \conv_x(\Labelnode)\cap\conv_x(\Labelnode') = \emptyset,
\end{equation*}
with the final equality following from the fact that~$x\in\conf(X,\cellMTree)$. Lastly, consider the case where either~$\Labelnode$ or~$\Labelnode'$ is an ancestor of~$\Labelleaf_i$, but not both. Without loss of generality assume~$\Labelnode'$ is an ancestor of~$\Labelleaf_i$ and~$\Labelnode$ is not. Note that
\begin{equation*}
    \conv_{x(p)}(\Labelnode')\subseteq \conv_x(\Labelnode')\cup \ShortPath(x_i,p) \subseteq  \conv_x(\Labelnode')\cup P.
\end{equation*}
Thus, since~$\conv_x(\Labelnode)\cap\conv_x(\Labelnode') = \emptyset$, we have
\begin{equation*}
    \conv_{x(p)}(\Labelnode)\cap \conv_{x(p)}(\Labelnode') \subseteq \conv_x(\Labelnode)\cap \big(\conv_x(\Labelnode')\cup P\big) = \conv_x(\Labelnode)\cap P = \emptyset.
\end{equation*}
Thus the criteria for~$x(p)$ to be an element of~$\conf(X,\cellMTree)$ are satisfied for all~$p\in P$. The result follows.
\end{proof}

\begin{lemma}
\label{lem:moveoffbranch}
Let~$b$ be a branch point or leaf in~$X$ with incident edge~$e$, and~$x = (x_1,\ldots, x_n) \in \conf(\tree,\cellMTree)$. If~$x_i = b$ for some~$i$, then sufficiently short paths from~$x_i$ to points~$y$ in the interior of~$e$ define paths from~$x$ to~$(x_1,\ldots, x_{i-1},y,x_{i+1}, \ldots, x_n)$ in~$\conf(\tree,\cellMTree)$.
\end{lemma}

\begin{proof}
Let~$y$ be sufficiently close to~$x_i=b$ that there is no~$x_j$, here~$j\neq i$, and no additional branch point in the shortest path from~$y$ to~$x_i$. Denote this path by~$P$ and let~$x' = (x_1,\ldots, x_{i-1},y,x_{i+1}, \ldots, x_n)$. For~$\Labelnode\in \cellMTree$ a node not ancestral to~$\Labelleaf_i$, the set~$\conv_x(v)$ does not intersect~$P$, as doing so would mean it would contain~$x_i$. The result follows by Lemma \ref{lem:moveonepoint}.
\end{proof}

\begin{lemma}
\label{lem:gatherpoints}
Fix any~$x = (x_1,\ldots,x_n)\in \conf(\tree,\cellMTree)$. There is an edge~$e$ in~$X$ such that there is a path from~$x$ to~$x' = (x'_1,\ldots,x'_n)$ in~$\conf(\tree,\cellMTree)$ where every~$x'_i$ is in the interior of~$e$.
\end{lemma}

\begin{proof}
Using Lemma~\ref{lem:moveoffbranch}, we assume without loss of generality that no point of~$x$ is on a branch point or leaf of~$X$. Let~$e\in X$ be an edge with some but not all entries of~$x$ in its interior. We will build a path to an~$x'$ with one more entry in~$e$, giving us the lemma by induction. The construction of the path is illustrated in Figure \ref{fig:exbringtoedgelemma}. 

By hypothesis, there is an entry of~$x$ in at least one of the two path components of~$X$ minus the interior of~$e$. Let~$b$ denote the endpoint of~$e$ in this path component, which we will call~$C$. Let~$x_j$ denote the entry of~$x$ in the interior of~$e$ that is closest to~$b$. We will construct a path from~$b$ to some~$x_i\in C$ and show that this path lifts to a path in~$\conf(\tree,\cellMTree)$, using Lemma \ref{lem:moveonepoint}.

\begin{figure}[htbp]
\centering
\resizebox{.8\textwidth}{!}{
\includegraphics{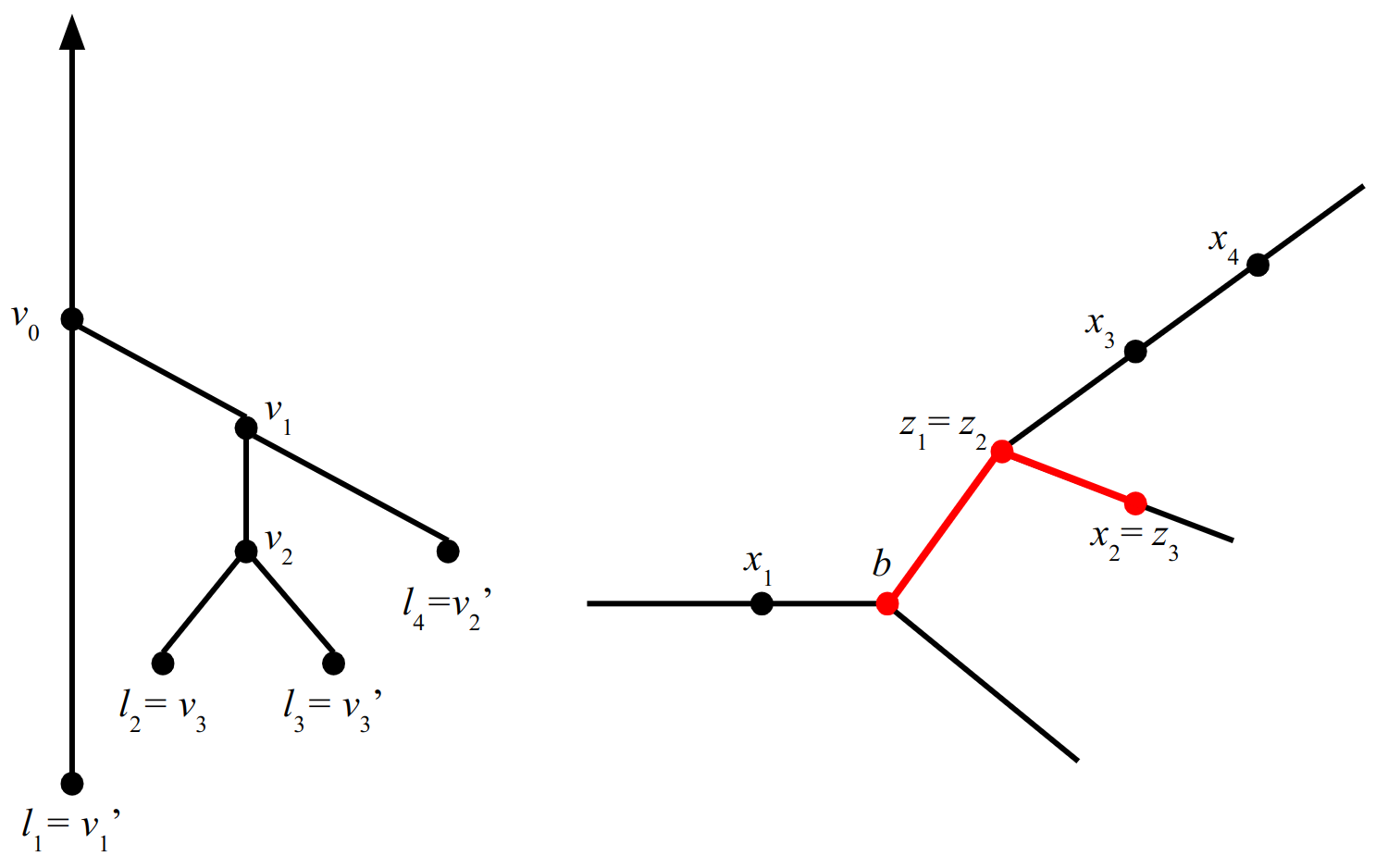}
}
    \caption{A merge tree~$\cellMTree$ (left) and an example geometric tree $X$ (right) illustrating the main construction of Lemma~\ref{lem:gatherpoints}. Highlighted in red is the path constructed from from a point~$x_i$ in~$X$ to~$b$, which lifts to a path in~$\conf(\tree,\cellMTree)$. Here,~$x_1$ plays the role of~$x_j$ in the proof. Note that in this example we could have also constructed a path from~$b$ to~$x_3$, but not to~$x_4$, because~$\conv_x(v_2)$ interrupts the path from~$b$ to~$\conv_x(v_2')$.} 
    \label{fig:exbringtoedgelemma}
\end{figure}

Let~$\Labelnode_0\in \cellMTree$ be the least ancestor of~$\Labelleaf_j$ such that~$\conv_x(\Labelnode_0)$ contains~$b$. Thus~$\Labelnode_0$ is not~$\Labelleaf_j$ itself, so~$\Labelnode_0$ has two children,~$\Labelnode_1$ and~$\Labelnode'_1$, exactly one of which is an ancestor of~$\Labelleaf_j$. Without loss of generality, suppose~$\Labelnode'_1$ is an ancestor of~$\Labelleaf_j$ and~$\Labelnode_1$ is not. Let~$P$ denote the shortest path from~$b$ to~$\conv_x(\Labelnode_1)$, and let~$z_1$ denote the endpoint of~$P$ intersecting~$\conv_x(\Labelnode_1)$. Since both~$b$ and~$z_1$ are in~$\conv_x(\Labelnode_0)$, so is~$P$. If~$b$ is in~$\conv_x(\Labelnode_1)$, then~$P$ consists of the singleton~$b$ and does not intersect~$\conv_x(\Labelnode'_1)$. Otherwise,~$b$ is in neither~$\conv_x(\Labelnode_1)$ nor~$\conv_x(\Labelnode'_1)$. Thus~$b$ lies on the shortest path between these two sets. Hence~$P$ is a subset of this shortest path, and therefore does not intersect~$\conv_x(\Labelnode'_1)$. So $P$ does not intersect $\conv_x(\Labelnode'_1)$ in either case.

We will continue to augment~$P$ until it has reached some~$x_i$. Inductively, assume we have already constructed a path~$P$ from~$b$ to a point~$z_{m-1}$ in~$\conv_x(\Labelnode_{m-1})$. If~$\Labelnode_{m-1}$ has no children, then~$\conv_{x}(\Labelnode_{m-1})$ is a singleton containing~$x_i$ for some~$i$, and we are done. Otherwise, let~$\Labelnode_m$ and~$\Labelnode'_m$ denote the children of~$\Labelnode_{m-1}$.

Either the shortest path from~$z_{m-1}$ to~$\conv_x(\Labelnode_m)$ does not intersect~$\conv_x(\Labelnode'_m)$ or the shortest path from~$z_{m-1}$ to~$\conv_x(\Labelnode'_m)$ does not intersect~$\conv_x(\Labelnode_m)$. Without loss of generality, assume we are in the first case. Since~$z_{m-1}$ is in~$\conv_x(\Labelnode_{m-1})$, so is the shortest path from~$z_{m-1}$ to~$\conv_x(\Labelnode_m)$. We augment~$P$ by this path and refer to its augmented endpoint as~$z_m$. Since~$\Labelnode_0$ has finitely many descendants, this process is guaranteed to terminate eventually, and we will obtain a path~$P$ from~$b$ to some~$x_i$. Further, since every augmented portion of~$P$ lies in~$\conv_x(\Labelnode_m)$ for some~$m$, all of~$P$ lies in~$\conv_x(\Labelnode_0)$.

Let~$\Labelnode$ bet a node of~$\cellMTree$ not ancestral to~$l_j$. If~$\conv_x(\Labelnode)$ intersects~$P$ then it intersects~$\conv_x(\Labelnode_0)$, which contains~$P$. Therefore either~$\Labelnode\preceq \Labelnode_0$ or~$\Labelnode_0\preceq \Labelnode$. It cannot be that~$\Labelnode_0\preceq \Labelnode$ since~$\Labelleaf_i$ is not a descendant of~$\Labelnode$, so~$\Labelnode\preceq \Labelnode_0$. Again since~$\Labelleaf_i$ is not a descendant of~$\Labelnode$,~$\Labelnode$ cannot be~$\Labelnode_m$ for any~$m$, so~$\Labelnode$ is a descendant of some~$\Labelnode'_m$. However, by construction,~$P$ does not intersect~$\conv_x(\Labelnode'_m)$ for any~$m$, so it cannot intersect~$\conv_x(\Labelnode)$ either. Applying Lemma \ref{lem:moveonepoint} allows us to move~$x_i$ onto~$b$. Then applying Lemma~\ref{lem:moveoffbranch}, we can further move~$x_i$ into the interior of~$e$.
\end{proof}

Now that we know we can get all points onto one edge, we want to be able to move groups of points along curves. Let~$x = (x_1,\ldots,x_n)\in\conf_n(X)$, here~$\conf_n(X)$ is the usual ordered configuration space of~$n$ points on~$X$, and suppose that there is a subset~$\subtree\subseteq X$ homeomorphic via a map~$h$ to the unit interval~$[0,1]$, such that every~$x_i$ lies in~$\subtree$. The coordinates of~$x$ thus inherit a total order via the total order of their images under~$h$. Thus for some permutation~$\sigma$ of~$\{1,\ldots,n\}$, the inherited total order has form~$x_{\sigma(1)} \leq \ldots \leq x_{\sigma(n)}$. We refer to~$\sigma$ as the \emph{$h$-permutation of~$x$}. Notice that~$\sigma$ only depends on the orientation determined by~$h$.

\begin{lemma}
Let~$x \in\conf(\tree,\cellMTree)$,~$y\in \conf_n(\tree)$ and suppose that there is a subset~$\subtree\subseteq \tree$ homeomorphic via a map~$h$ to the unit interval~$[0,1]$, such that every coordinate of~$x$ and~$y$ lies in~$\subtree$. Then there is a path from~$x$ to~$y$ in~$\conf(\tree,\cellMTree)$ if~$x$ and~$y$ have the same~$h$-permutation. 
\label{lem:linetravel}
\end{lemma}

\begin{proof}
Let~$h$ and~$h^{-1}$ induce maps on~$\conf_n(Y)$ and~$\conf_n([0,1])$ by acting component-wise. Consider the path in~$\conf_n(X)$ from~$x$ to~$y$
\begin{equation*}
    \gamma(t) := h^{-1}\big[(1-t)h(x) + th(y)\big]
\end{equation*}
whose image under~$h$ linearly interpolates between~$h(x)$ and~$h(y)$. Thus the~$h$-permutation of~$\gamma(t)$ is the same as that of~$x$ for all~$t\in[0,1]$. Suppose~$\conv_{\gamma(t)}(\Labelnode)\cap\conv_{\gamma(t)}(\Labelnode')$ is nonempty. Then either there exists~$i$,~$j$, and~$k$ such that~$\Labelleaf_i$ and~$\Labelleaf_k$ are descendants of~$\Labelnode$,~$\Labelleaf_j$ is a descendent of~$\Labelnode'$ and~$\gamma(t)_i\leq \gamma(t)_j\leq \gamma(t)_k$ or the same is true with the roles of~$\Labelnode$ and~$\Labelnode'$ reversed. Without loss of generality suppose we are in the first case. Thus~$x_i\leq x_j \leq x_k$, since the~$h$-permutation is constant along~$\gamma$. So~$\conv_x(\Labelnode)\cap\conv_x(\Labelnode')$ is also nonempty and~$\Labelnode\preceq \Labelnode'$ or~$\Labelnode' \preceq \Labelnode$. Hence~$\gamma$ is a path in~$\conf(\tree,\cellMTree)$.
\end{proof}

\begin{lemma}
\label{lem:lineconfigrule}
Suppose~$x\in \conf(\tree,\cellMTree)$ consists of points on a subset~$\subtree$ of~$X$ homeomorphic to the unit interval via a map~$h$. Then there exists~$1\leq i<k \leq n$ such that~$\Labelleaf_{\sigma(i)} , \ldots, \Labelleaf_{\sigma(k)}$ are the descendants of~$\Labelnode$.
\end{lemma}

\begin{proof}
It suffices to show that the set~$\{j \mid \Labelleaf_{\sigma(j)}\preceq \Labelnode\}$ is convex. Consider two leaves~$\Labelleaf_{\sigma(i)}$ and~$\Labelleaf_{\sigma(k)}$, with~$i<k$, that have~$\Labelnode$ as ancestor, and let~$j\in \{i,\cdots, k\}$. Since~$x_{\sigma(i)}<x_{\sigma(j)}<x_{\sigma(k)}$, we have~$x_{\sigma(j)}\in \conv(x_{\sigma(i)},x_{\sigma(k)})\subseteq \conv_x(\Labelnode)$. Since~$\conv_x(\Labelleaf_{\sigma(j)}) = \{x_{\sigma(j)}\}$,~$\conv_x(\Labelleaf_{\sigma(j)})$ intersects~$\conv_x(\Labelnode)$. Therefore either~$\Labelleaf_{\sigma(j)} \preceq \Labelnode$ or~$\Labelnode \preceq \Labelleaf_{\sigma(j)}$. But~$\Labelleaf_{\sigma(j)}$ is a leaf so it must be the case that~$\Labelleaf_{\sigma(j)} \preceq \Labelnode$.
\end{proof}

This lemma implies that the homeomorphism~$h$ of the edge~$e$ where a configuration~$x \in \conf(\tree,\cellMTree)$ lies determines, for each internal node~$\Labelnode\in \cellMTree$, its left child~$\leftChild \Labelnode$ and right child~$\rightChild \Labelnode$. Explicitly, we can choose $\leftChild \Labelnode$ and $\rightChild 
\Labelnode$ to be such that $i<j$ whenever $\Labelleaf_{\sigma(i)}\preceq \leftChild\Labelnode$ and $\Labelleaf_{\sigma(j)} \preceq \rightChild\Labelnode$. Therefore, the configuration~$x$ induces the structure~$\cellMTree_x$ of a {\em chiral merge tree} on~$\cellMTree$, as defined in~\cite[Definition 5.3]{curry2018fiber}:
\begin{definition}
\label{definition_chiral_merge_tree}
    A {\em chiral merge tree} is a binary cellular merge tree where the two children of any internal node are labelled as either the {\em left} or {\em right} child.
\end{definition}
The next lemma will let us alter the chiral merge tree structure assigned to a configuration, when~$\tree$ is especially simple.

\begin{lemma}
\label{lem:starreconfigure}
Let~$\tree$ be a geometric starlike tree of degree 3, i.e.~$\tree$ is homeomorphic to three copies of~$[0,1]$ identified at~$0$. Let~$x\in \conf(\tree,\cellMTree)$ be a configuration lying on the interior of an edge~$e$ of~$X$. Fix a homeomorphism~$h$ of~$e$ with~$[0,1]$. Let~$\cellMTree_c$ be any chiral merge tree structure on~$\cellMTree$. Then there exists a path from~$x$ to some~$y\in \conf(\tree,\cellMTree)$ lying on the interior of~$e$, such that~$\cellMTree_y=\cellMTree_c$.
\end{lemma}

\begin{proof}
By induction, it is sufficient to prove the case where~$\cellMTree_c$ differs from~$\cellMTree_x$ by only inverting the left child~$\leftChild \Labelnode$ and right child~$\rightChild \Labelnode$ of a given node~$\Labelnode$. 

For simplicity, assume that the permutation~$\sigma$ induced by~$x$ and~$h$ is the identity, thus~$x_1 \leq \ldots \leq x_n$. Denote by~$\Labelleaf_{i},\ldots,\Labelleaf_{k}$ the descendants of~$\Labelnode$. There is some~$i\leq j < k$ such that~$\Labelleaf_{i},\ldots,\Labelleaf_{j}$ are the descendants of~$\leftChild \Labelnode$ while~$\Labelleaf_{j+1},\ldots,\Labelleaf_{k}$ are the descendants of~$\rightChild \Labelnode$. The goal is thus to build a path from~$x$ to some~$y = (y_1,\ldots,y_n)$ in~$\conf(\tree,\cellMTree)$, where each entry of~$y$ is interior to~$e$, satisfying
\begin{equation*}
    y_{1} \leq \ldots \leq y_{i-1} \leq y_{j+1} \leq \ldots \leq y_{k} \leq y_{i}\leq \ldots \leq y_{j} \leq y_{k+1} \leq \ldots \leq y_{n}.
\end{equation*}
Let~$e_1=e$,~$e_2$, and~$e_3$ be the edges of~$X$, and~$b$ be the branch point of~$X$. Without loss of generality assume that~$h$ is such that~$h(b) = 0$. We will show that the following sequence of moves from~$x$ to~$y$ are allowable in~$\conf(\tree,\cellMTree)$:
\begin{enumerate}
    \item Move the~$1^{\text{st}}$ through~$j^{\text{th}}$ coordinates from~$e_1$ into~$e_2$, in that order. 
    \item Move the~$(j+1)^{\text{th}}$ through~$k^{\text{th}}$ coordinates from~$e_1$ into~$e_3$, in that order.
    \item Move the~$i^{\text{th}}$ through~$j^{\text{th}}$ coordinates from~$e_2$ into~$e_1$, in reverse order.
    \item Move the~$(j+1)^{\text{th}}$ through~$k^{\text{th}}$ coordinates from~$e_3$ into~$e_1$, in reverse order.
    \item Move the~$1^{\text{st}}$ through~$i^{\text{th}}$ coordinates from~$e_2$ into~$e_1$, in reverse order.
\end{enumerate}
Figure \ref{fig:reconfigure} shows a visualization of this sequence of moves in an example where there are five coordinates.

\begin{figure}[htbp]
\centering
\resizebox{\textwidth}{!}{
\includegraphics{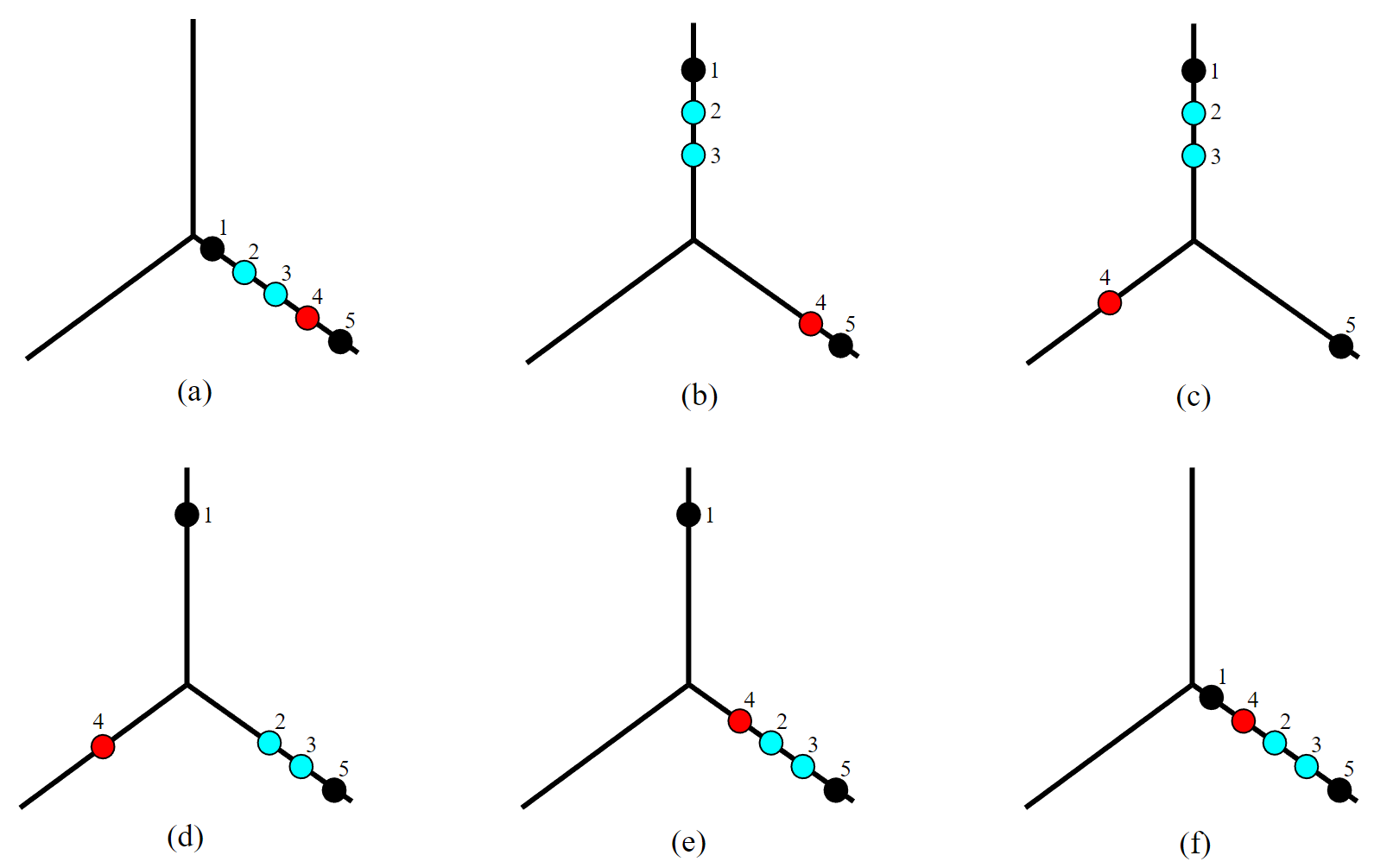}
}
    \caption{A visual representation of the proof of Lemma \ref{lem:starreconfigure}. In the depicted example,~$\cellMTree$ has five leaves, three of which are descendant from~$\Labelnode$. One of these nodes, highlighted in red, is moreover a descendant of~$\leftChild\Labelnode$. The other two descendants of~$\Labelnode$, highlighted in blue, are descendants of~$\rightChild \Labelnode$. Panels (a) through (f) show the path used to reconfigure points in the proof.}
    \label{fig:reconfigure}
\end{figure}
Moves 1 and 5 can be realized by a path in~$\conf(\tree,\cellMTree)$ via Lemma \ref{lem:linetravel} with~$\subtree = e_1 \cup e_2$. For move~2, suppose we have already moved coordinates~$j+1$ through~$m-1$ into~$e_3$ to attain a configuration~$z = (z_1,\ldots, z_n)$. Let~$P$ denote the shortest path from~$z_m$ to~$b$ and~$z' = (z_1,\ldots,z_{m-1}, b,z_{m+1}, \ldots, z_n)$. Let~$\Labelnode_0$ be a node of~$\cellMTree$ not ancestral to~$\Labelleaf_m$.

Suppose that~$\conv_z(\Labelnode_0)$ intersects~$P$. Since~$\Labelleaf_m \npreceq \Labelnode_0$,~$\conv_z(\Labelnode_0)$ can only intersect~$P$ at~$b$. This means that~$\conv_{z'}(\Labelnode_0)$ intersects the interiors of both~$e_2$ and~$e_3$, so there must be points on~$e_3$ in~$z$ already, i.e.~$m>j+1$. In particular,~$\conv_z(\Labelnode_0)$ contains both~$z_j$ and~$z_{m-1}$. Since~$\Labelleaf_j$ and~$\Labelleaf_{m-1}$ are descendants of~$\leftChild\Labelnode$ and~$\rightChild\Labelnode$ respectively,~$\Labelnode_0$ must be an ancestor of~$\Labelnode$. Since~$\Labelleaf_m$ is a descendant of~$\Labelnode$,~$\Labelnode_0$ is an ancestor of~$\Labelleaf_m$, a contradiction. Hence~$\conv_z(\Labelnode)$ does not intersect intersect~$P$, and so by Lemma \ref{lem:moveonepoint} we can move the~$m^{\mathrm{th}}$ coordinate of~$z$ to~$b$. Then applying Lemma \ref{lem:moveoffbranch} we can move the~$m^{\mathrm{th}}$ coordinate of~$z$ into the interior of~$e_3$. Induction on~$m$ then allows us to complete move 2.

The cases of moves 3 and 4 are handled similarly.
\end{proof}

\begin{lemma}
\label{lemma_connected_config_space_starlike_tree}
Let~$X$ be a geometric starlike tree of degree 3. Then~$\conf(\tree,\cellMTree)$ is path-connected.
\end{lemma}

\begin{proof}
Fix an edge~$e\in \tree$ with an orientation~$h$. By Proposition~\ref{prop_fiber_merge_tree_non_empty},~$\conf(\tree,\cellMTree)\neq \emptyset$. Let~$x,y\in \conf(\tree,\cellMTree)$. Up to applying Lemma~\ref{lem:gatherpoints} and Lemma~\ref{lem:linetravel}, we can assume that~$x$ lies in the interior of~$e$. Similarly we may assume for~$y$ lies in the interior of~$e$. 

By Lemma~\ref{lem:starreconfigure}, we can connect~$x$ to a configuration~$x'\in \conf(\tree,\cellMTree)$ lying on~$e$ and such that~$\cellMTree_{x'}=\cellMTree_{y}$. In particular~$x'$ and~$y$ induce the same~$h$-permutation, and therefore, thanks to Lemma~\ref{lem:linetravel}, there is a path between them in~$\conf(\tree,\cellMTree)$. 
\end{proof}

Finally we can prove the central proposition of the section, restated below for convenience.

\begin{repproposition}{proposition_conf_space_connected}
Let~$X$ be a tree not homeomorphic to the unit interval. Let~$\cellMTree$ be a generic cellular merge tree. Then~$\conf(X,\cellMTree)$ is path-connected. 
\end{repproposition}

\begin{proof}[Proof of Proposition~\ref{proposition_conf_space_connected}]
Let~$l$ be a leaf of~$X$, and~$e$ be the edge incident to~$l$. Since~$\tree$ is connected and not homeomorphic to the unit interval, the other endpoint~$b$ of~$e$ must be a branch point. Hence, there is a subtree~$\subtree \subseteq \tree$ which is a geometric starlike tree of degree 3 containing~$e$. By Proposition~\ref{prop_fiber_merge_tree_non_empty},~$\conf(\subtree,\cellMTree)\neq \emptyset$. Let~$y\in\conf(\subtree,\cellMTree)\subseteq \conf(\tree,\cellMTree)$ be a fixed, target configuration on~$\tree$.

Let~$x = (x_1,\ldots,x_n)\in \conf(\tree,\cellMTree)$. Applying Lemma \ref{lem:gatherpoints} and then Lemma~\ref{lem:linetravel}, we find a path in~$ \conf(\tree,\cellMTree)$ from~$x$ to a configuration~$x'$ whose points lie in the interior of~$e$. Viewing~$x'$ as a configuration in~$\conf(\subtree,\cellMTree)\subseteq \conf(\tree,\cellMTree)$, by~Lemma~\ref{lemma_connected_config_space_starlike_tree} it can be joined to~$y$ via a path in~$\conf(\subtree,\cellMTree)$, which also defines a path in~$ \conf(\tree,\cellMTree)$. 
\end{proof}

\begin{corollary}
\label{corollary_distance_connected_comp}
Let~$X$ be a tree not homeomorphic to the unit interval. Let~$\delta_L$ be the minimum distance between pairs of non-equal interval left endpoints in~$D$ and~$\delta_R$ be the minimum distance between pairs of non-equal right endpoints. Then the path connected components of the fiber~$\fib$ are at distance at least~$\min(\delta_L,\delta_R)$ from each other. In particular, the path connected components of~$\fib$ are the connected components of~$\fib$.
\end{corollary}

\begin{proof}
Let~$f,g:X\to \R$ be functions in distinct path connected components of the fiber~$\fib$. Then by Theorem~\ref{theorem_fiber_merge_tree_connected}, their merge trees~$\MT(f)$ and~$\MT(g)$ are non-isomorphic, and therefore by Proposition~\ref{prop_small_distance_functions_implies_same_merge_trees}, we have that~$\|f-g\|_\infty\geq\min(\delta_L,\delta_R)$.
\end{proof}

\begin{corollary}
\label{corollary_counting_functions_in_the_fiber}
Let~$X$ be a tree not homeomorphic to the unit interval. The fiber~$\fib$ has a finite number of connected components given by:
\[\# \pi_0(\fib)=\prod_{[b,d)\in D} \# \big \{ [b',d') \in D \mid  [b,d) \subset [b',d')\big\}.\]
\end{corollary}

\begin{proof}
This is the number of distinct cellular merge trees with barcode~$D$, see~\cite[Theorem 4.8]{curry2018fiber}. From Theorem~\ref{theorem_fiber_merge_tree_connected}, such merge trees are in bijection with path connected components of~$\fib$, which equal connected components of~$\fib$ by Corollary~\ref{corollary_counting_functions_in_the_fiber}.
\end{proof}

\subsection{Topology of connected components in the fiber}
\label{subsection_topology_connected_components}

When there are very few leaves in the cellular merge tree~$\cellMTree$, and hence very few points in~$\conf(X,\cellMTree)$, we are able to deduce the homotopy type of the connected components of~$\MT^{-1}(\cellMTree)$, and hence~$\fib$ for simple barcodes~$D$. The simplest case is when~$\cellMTree$ has only one leaf.

\begin{corollary}
\label{corollary_fiber_barcode_1_interval}
Let~$X$ be the geometric realization of a tree,~$\cellMTree$ be a merge tree with one leaf and~$D$ be the barcode associated to~$\cellMTree$. Then~$\MT^{-1}(\cellMTree) = \fib$ and both are contractible.
\end{corollary}

\begin{proof}
The barcode~$D$ consists of one infinte interval~$[a,\infty )$, where~$a$ is the value assigned to the one leaf of~$\cellMTree$. The only merge tree that can give rise to~$D$ is~$\cellMTree$, by \cite[Theorem 4.8]{curry2018fiber}. Hence~$\MT^{-1}(\cellMTree) = \fib$. By Theorem \ref{theorem_homotopy_equivalence_configuration_spaces},
\begin{equation*}
    \MT^{-1}(\cellMTree) \simeq \conf(X,\cellMTree) = \conf_1(X) = X,
\end{equation*}
and~$X$ is contractible, so~$\MT^{-1}(\cellMTree)$ is contractible.
\end{proof}

If there are only two points in~$\conf(X,\cellMTree)$, its structure is still fairly simple, and has already been computed up to homotopy. We derive the following as an immediate consequence.

\begin{corollary}
\label{corollary_fiber_barcode_2_intervals}
Let~$X$ be the geometric realization of a tree with at least one vertex of degree~$\geq 3$ and~$\cellMTree$ be a cellular merge tree with exactly two leaves~$l_1$ and~$l_2$. Suppose~$\pi(l_1) \neq \pi(l_2)$. Then~$\MT^{-1}(\cellMTree)$ is homotopy equivalent to the wedge sum of
\begin{equation*}
    -1 + \sum_{v\in N(X)} (\eta(v) - 1)(\eta(v) - 2)
\end{equation*}
circles, where~$N(X)$ denotes the nodes in any cellular decomposition of~$X$ and~$\eta(v)$ denotes the degree of node~$v$ in~$X$. Moreover, denoting by~$D$ the barcode arising from the merge tree~$\cellMTree$, we have~$\fib = \MT^{-1}(\cellMTree)$.
\end{corollary}

\begin{proof}
Theorem~\ref{theorem_homotopy_equivalence_configuration_spaces} tells us that~$\MT^{-1}(\cellMTree)$ is homotopy equivalent to~$\conf(\tree,\cellMTree)$, here~$\conf(\tree,\cellMTree)=\conf_2(X)$ is the configuration space of two points on~$X$, whose homotopy type is computed in~\cite[Theorem 11.1]{farber2018configuration}. The last statement follows from~\cite[Theorem 4.8]{curry2018fiber}:~$D$ is the barcode with two intervals, one finite contained by the other, infinite interval, and therefore~$\cellMTree$ is the only merge tree giving rise to the barcode~$D$.
\end{proof}

\begin{remark}
Let~$X$ be the star-like tree made of~$n$ edges joined at one vertex, and let~$D$ be a barcode as in Corollary~\ref{corollary_fiber_barcode_2_intervals}. Then the corollary tells us that~$\fib$ is homotopy equivalent to a wedge of~$n^2 -3n +1$ circles.

In \cite{mischaikow2021persistent}, the authors consider the same barcode~$D$ and the discrete tree~$Y$ obtained from~$X$ by inserting an additional vertex in the middle of each edge. Consider the space of functions~$f$ on the vertices and edges of~$Y$, where~$f(e) = \max(f(v),f(w))$ for any edge~$e = (v,w)$. In this scenario, the authors find that~$\fib$ is also homotopy equivalent to a wedge of~$n^2-3n+1$ circles \cite[Theorem~8.1]{mischaikow2021persistent}. This suggests that there may be a general relationship between the fiber of persistent homology of geometric trees and discrete trees with sufficiently fine triangulations.
\end{remark}

\subsection*{Acknowledgements}
DB is a member of the Centre for Topological Data Analysis, funded in part by EPSRC EP/R018472/1. JL was supported by an LMS Early Career Fellowship and in the last stage of the project by St John's College for his research assistantship role with Heather Harrington.

\bibliographystyle{abbrv}
\bibliography{refs}

\end{document}